\documentclass[10 pt,a4 paper]{article}
\usepackage[utf8]{inputenc}
\usepackage[OT1]{fontenc}
\usepackage[english]{babel}
\usepackage{dsfont}
\usepackage{amsfonts}
\usepackage{amsmath}
\usepackage{wasysym}
\usepackage{amsthm}
\usepackage{amssymb}
\usepackage{amsthm}
\usepackage{float}
\usepackage{textcomp}
\usepackage{xcolor}
\usepackage{graphicx}
\usepackage{wrapfig}
\usepackage{tikz}
\usepackage{mathdots}
\usepackage{fancybox}
\usepackage{fancyhdr}
\usepackage{mathrsfs}
\usepackage{comment}
\usepackage{array}

\usepackage[all]{xy}
\usepackage{centernot}
\usepackage{listings}
\usepackage{faktor}
\usepackage{bbm}
\usepackage{bm}
\usepackage[linktoc=page]{hyperref}
\hypersetup{colorlinks=true,linkcolor=red}

\usepackage{newlfont}
\usepackage{geometry}
\usepackage{ccaption}
\usetikzlibrary{matrix,arrows,decorations.pathmorphing, calc}

\usepackage[backend=biber,style=alphabetic,sorting=nyt,maxbibnames=9]{biblatex}
\renewbibmacro{in:}{%
	\ifentrytype{article}{}{\printtext{\bibstring{in}\intitlepunct}}}

\usepackage{guit}
\usepackage{calligra,amsmath,amssymb}
\usepackage{calc}
\usepackage{lmodern}

\usepackage{ytableau}
\usepackage{subcaption}
\usepackage{tocloft}

\usepackage{hhline}

\usepackage{dynkin-diagrams}

\newcounter{daggerfootnote}

\newtheorem{prop}{Proposition}[section]
\newtheorem{lemma}[prop]{Lemma}
\newtheorem{cor}[prop]{Corollary}
\newtheorem{teo}[prop]{Theorem}
\newtheorem*{main-teo}{Main Theorem}
\newtheorem{conjecture}[prop]{Conjecture}

\theoremstyle{definition}
\newtheorem{Def}[prop]{Definition}
\newtheorem*{Def*}{Definition}
\newtheorem*{prop*}{Proposition}
\newtheorem*{lema*}{Lemma}

\newtheorem{obs}[prop]{Remark}
\newtheorem*{obs*}{Remark}
\newtheorem*{cor*}{Corollary}
\newtheorem*{teo*}{Theorem}
\newtheorem{es}[prop]{Example}
\newtheorem*{es*}{Example}

\numberwithin{equation}{section}

\DeclareMathOperator{\Ker}{ker}

\DeclareMathOperator{\im}{Im}

\DeclareMathOperator{\Hom}{Hom}

\DeclareMathOperator{\Stab}{stab}

\DeclareMathOperator{\GL}{GL}
\DeclareMathOperator{\rk}{Rk}

\DeclareMathOperator{\Sym}{Sym}

\DeclareMathOperator{\Sing}{Sing}

\DeclareMathOperator{\SL}{SL}

\DeclareMathOperator{\codim}{codim}

\DeclareMathOperator{\Gr}{Gr}

\DeclareMathOperator{\Mat}{Mat}

\DeclareMathOperator{\SO}{SO}

\DeclareMathOperator{\Q}{Q}

\DeclareMathOperator{\Spin}{Spin}

\DeclareMathOperator{\Cl}{Cl}
\DeclareMathOperator{\conj}{conj}
\DeclareMathOperator{\Pf}{Pf}
\DeclareMathOperator{\diam}{diam}

\DeclareMathOperator{\OG}{OGr}
\DeclareMathOperator{\Res}{Res}

\DeclareMathOperator{\Terr}{Terr}

\DeclareMathOperator{\Sp}{Sp}

\newcolumntype{P}[1]{>{\centering\arraybackslash}p{#1}}

\definecolor{coloreteorema}{HTML}{E8EDFA}
    {\endMakeFramed}

    {\endMakeFramed}

\title{Identifiability and singular locus\\ of secant varieties to spinor varieties}
\author{Vincenzo Galgano\footnote{Dipartimento di Matematica, Università di Trento, Via Sommarive 14 38123 Trento, Italy; ORCID: 0000-0001-8778-575X (\texttt{vincenzo.galgano@unitn.it})}}
\date{}

\begin{document}

\maketitle

\begin{abstract}
	\noindent In this work we analyze the $Spin$--structure of the secant variety of lines $\sigma_2(\mathbb S)$ to a Spinor variety $\mathbb S$ minimally embedded in its spin representation. In particular, we determine the poset of the $Spin$--orbits and their dimensions. We use it for solving the problems of identifiability and tangential-identifiability in $\sigma_2(\mathbb S)$, and for determining the second Terracini locus of $\mathbb S$. Finally, we show that the singular locus $Sing(\sigma_2(\mathbb S))$ contains the two $Spin$--orbits of lowest dimensions and it lies in the tangential variety $\tau(\mathbb S)$: we also conjecture what it set-theoretically is.\\
	\hfill\break
	\hfill\break
\end{abstract}

\tableofcontents

\section*{Acknowledgement}

This work is part of the author's Ph.D. thesis supervised by Alessandra Bernardi and Giorgio Ottaviani, whom the author sincerely thanks for their patient guidance and precious discussions. The author also thanks Laurent Manivel and Mateusz Micha\l{}ek for their helpful suggestions. The author is member of the italian INdAM group GNSAGA.

\section*{Introduction}

\indent \indent Given an irreducible non-degenerate projective variety $X\subset \mathbb P^M_\mathbb C$ and a point $f \in \mathbb P^M_\mathbb C$, the {\em $X$--rank} of $f$ with respect to $X$ is the minimum number of points of $X$ whose linear span contains $f$. In particular, point of $X$ have $X$--rank $1$. Then, starting from $X$, for any $r\in \mathbb Z_{>0}$ one defines the {\em $r$--th secant variety} $\sigma_r(X)\subset \mathbb P^M_\mathbb C$ as the Zariski closure of points of $\mathbb P^M_\mathbb C$ of $X$--rank at most $r$.\\
\indent Secant varieties have been studied for decades, but several aspects of their geometry are still mysterious and difficult to handle with. Even computing their dimensions is a hard task and a current topic of research. Two crucial aspects of secant varieties which are fundamental for applications and are still unknown in general are the {\em identifiability} and the {\em singularity} of their points. Namely, a given point $f\in \mathbb P^M_\mathbb C$ of $X$--rank $r$ is identifiable if it admits a unique decomposition as sum of $r$ elements of $X$. On the other hand, singular points in $\sigma_r(X)$ are those such that the dimension of the tangent spaces at these points overcome the dimension of the secant variety: conversely, the points for which such dimensions coincide are smooth. Of course, identifiability and singularity are of great impact to both theory and applications. Results having the smoothness (ie. non-singularity) of the variety among the hypotheses are a dense in geometry: for instance, in \cite[Prop. 5.1]{COV} the authors give a criterion for identifiability of specific tensors in a secant variety under the assumption that these tensors are smooth. In fact, identifiability and smoothness are quite related each other and often one notion suggests the other (as in this work), although in general both implications admit counterexamples.\\
\indent Owned to the above, determining the {\em singular locus} $\Sing(\sigma_r(X))$ of a secant variety $\sigma_r(X)$ is a central problem. Classically, if $\sigma_r(X)$ is not a linear space, it is known that the singular locus $\Sing(\sigma_r(X))$ contains the secant variety $\sigma_{r-1}(X)$ but only in few cases $\Sing(\sigma_r(X))$ is actually determined. For instance, the case $r=2$ for {\em Segre varieties} is solved in \cite{michalek2015secant}; for {\em Veronese varieties} the cases $r=2$ and $r=3$ are solved in \cite{kanev1999chordal} and \cite{kangjin2018} respectively, while partial results for higher cases $r\geq 4$ are obtained in \cite{kangjin2021}; partial results for the $2$--nd secant variety of Grassmannians and other cominuscule varieties appear in \cite{manivelmichaleksecants}.\\
\indent Segre varieties, Veronese varieties and Grassmannians are example of projective rational homogeneous varieties (RHV). They are described as quotients $G/P$ of a semisimple (complex, in our case) Lie group $G$ by a parabolic subgroup $P$, or equivalently as unique closed orbits (of the highest weight vectors) into projectivized representations of such groups. The Representation Theory behind these varieties allows to derive several geometric properties: in this respect, the geometry of RHVs has been largely studied in \cite{LM03, LM04}. In particular, a wide literature has been devoted to the secant variety of lines $\sigma_2(G/P)$ and the tangential variety $\tau(G/P)$ (ie. unions of all tangent lines) to a RHV $G/P$, starting with Zak's key work \cite{zak1993tangents} and continuing with \cite{Kaj99, landsbergweymantangential,landsbergweymansecant,  manivelmichaleksecants, russo2016}. \\
\indent Another example of RHVs is the {\em Spinor varieties}, described by the quotients $\Spin(V)/P(v_{\omega})$ of a spin group $\Spin(V)$ for the parabolic subgroup $P(v_{\omega})$ stabilizing the highest weight vector $v_{\omega}$ in a (half-)spin representation of fundamental weight $\omega$. The (half-)spin representations are defined by the fundamental weights corresponding to the right-hand-side extremal nodes of the Dynkin diagrams of type $B_N$ and $D_N$. 
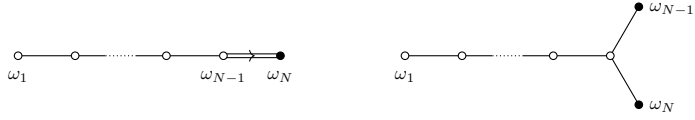
\[
\dynkin[labels={\omega_1,,,\omega_{N-1},\omega_{N}},edge length=.75cm] B{oo.oo*} 
\ \ \ \ \ \ \ \ \ \ 
\dynkin[labels={\omega_1,,,,\omega_{N-1},\omega_{N}},edge length=.75cm] D{oo.oo**}
\]
Despite their bijection with maximal orthogonal Grassmannians, they are not well understood as well as the classical Grassmannians. However, due to their relation with Clifford algebras, they have a richer structure making them an interesting topic of research. Moreover, Spinor varieties appear in Quantum Information as sets of separable states in fermionic Fock spaces \cite{holweckfermionic2015, holweckfermionic2018}. In general, equations and singular loci of secant varieties of lines to Spinor varieties are still not determined: the only completely understood case is the one of the Spinor variety defined by the spin group $\Spin_{12}$, as one of the {\em Legendrian varieties} \cite{LM01, LM07}. \\
\hfill\break
\indent The aim of this paper is to describe the secant variety of lines to a Spinor variety. Due to the fact that Spinor varieties of type $B_{N}$ and $D_{N+1}$ are projectively equivalent, throughout all this paper we work with the Spinor variety $\mathbb S_{N}^+$ lying in the projectivization of the half-spin representation $V_{\omega_N}^{D_N}=\bigwedge^{ev}\mathbb C^N$, under the assumption that $N$ is even  (cf. Remark \ref{rmk:setting}). The main idea of the work is to use the natural action of $\Spin_{2N}$ on $V_{\omega_N}^{D_N}$, which leaves $\sigma_2(\mathbb S_{N}^+)$ invariant. In particular, the singular locus $\Sing(\sigma_2(\mathbb S_{N}^+))$ is stable under the action of $\Spin_{2N}$, and it is the union of some orbits. Then one can check the singularity or smoothness of an orbit by simply checking it for a representative. \\
\indent In Theorem \ref{thm:orbit partition of sec} we prove that the poset of $\Spin_{2N}$--orbits in $\sigma_2(\mathbb S_{N}^+)$ is 
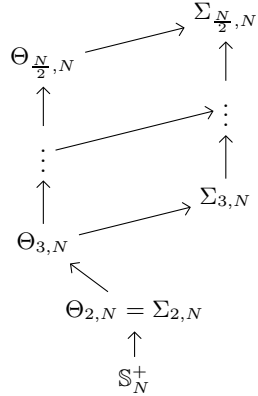
\begin{center}
	{\small \begin{tikzpicture}[scale=3]
		
		\node(S) at (0,0.4){{$\mathbb S_N^+$}};
		\node(t2) at (0,0.7){{$\Theta_{2,N} = \Sigma_{2,N}$}};
		\node(t3) at (-0.4,1){{$\Theta_{3,N}$}};
		\node(t) at (-0.4,1.4){{$\vdots$}};
		\node(td) at (-0.4,1.8){$\Theta_{\frac{N}{2},N}$};
		
		\node(s3) at (0.4,1.2){{$\Sigma_{3,N}$}};
		\node(s) at (0.4,1.6){{$\vdots$}};
		\node(sd) at (0.4,2){{$\Sigma_{\frac{N}{2},N}$}};
		
		\path[font=\scriptsize,>= angle 90]
		(S) edge [->] node [left] {} (t2)
		(t2) edge [->] node [left] {} (t3)
		(t3) edge [->] node [left] {} (s3)
		(t3) edge [->] node [left] {} (t)
		(t) edge [->] node [left] {} (td)
		(t) edge [->] node [left] {} (s)
		(td) edge [->] node [left] {} (sd)
		(s3) edge [->] node [left] {} (s)
		(s) edge [->] node [left] {} (sd);
		\end{tikzpicture}}
\end{center}
where the orbits $\Theta_{l,N}$ (cf.\ \eqref{def:tangent orbits}) are made of tangent points to $\mathbb S_N^+$, the orbits $\Sigma_{l,N}$ (cf.\ \eqref{def:secant orbits}) are made of points lying on bisecant lines to $\mathbb S_N^+$, and the arrows denote the inclusion of an orbit into the closure of another orbit.\\
\indent Via the theory of the nonabelian apolarity from \cite{LO13}, we introduce the {\em Clifford apolarity} (cf.\ Sec. \ref{sec:clifford apolarity}) which allow us to completely solve the problems of identifiability and {\em tangential-identifiability} (cf.\ Def.\ \ref{def:cactus}) of points in $\sigma_2(\mathbb S_N^+)$. More precisely, we prove the following result, collecting Proposition \ref{prop:contact locus for Sigma2} and Theorems \ref{thm:identifiable secant orbits}, \ref{thm:tangential-identifiable tangent orbit} at once. 
\begin{main-teo}
	Points in the orbit $\Sigma_{2,N}$ are not identifiable, while points in the orbits $\Sigma_{l,N}$ for $l\geq 3$ are identifiable. Moreover, points in the orbits $\Theta_{l,N}$ for $l\geq 3$ are tangential-identifiable.
\end{main-teo}
\indent We apply these results for computing the dimensions of the orbit closures. Moreover, we describe the second Terracini locus of a Spinor variety, whose importance in geenral relies in its strong relation with the singular locus. 
\begin{teo*}[Theorem \ref{thm:terracini locus spinor}]
	The second Terracini locus $\Terr_2(\mathbb S_N^+)$ of the Spinor variety $\mathbb S_N^+$ corresponds to the orbit closure $\overline{\Sigma_{2,N}}=\mathbb S_N^+ \sqcup \Sigma_{2,N}$.
\end{teo*}
\indent Finally, we obtain a lower and upper bound on the singular locus of $\sigma_2(\mathbb S_N^+)$.

\begin{teo*}[Corollary \ref{cor:sing locus spinor}]
	For any $N\geq 7$, the singular locus of the secant variety of lines $\sigma_2(\mathbb S_{N}^+)$ contains the orbit closure $\overline{\Sigma_{2,N}}$ and lies in the tangential variety $\tau(\mathbb S_N^+)$.
\end{teo*}

\noindent We also conjecture that the singular locus actually coincides with the above orbit closure.\\
\hfill\break
\indent We point out that the same poset of orbits and the same results on identifiability have been obtained in \cite{galganostaffolani2022grass} for the case of Grassmannians, simultaneously to this work, although with different techniques. This suggests a more general behaviour for a wider subclass of RHVs, namely the cominuscule varieties. However, we leave this generalization for future work.\\
\hfill\break
\indent The paper is organized as follows. In Sec.\ \ref{sec:preliminaries} we recall some basic facts and we fix our notation. Sec.\ \ref{sec:clifford apolarity} is devoted to the Clifford apolarity. In Sec.\ \ref{sec:poset spinor} we determine the poset of $\Spin_{2N}$--orbits in $\sigma_2(\mathbb S_{N}^+)$. In Secc.\ \ref{sec:identifiability spinor} and \ref{sec:tangential-identifiability} we apply the Cliffor apolarity for determining identifiable and tangential-identifiable points in $\sigma_2(\mathbb S_{N}^+)$: these results allow us to compute the dimensions of the orbit closures in Sec.\ \ref{sec:orbit dimensions spinor}, as well as to describe the second Terracini locus $\Terr_2(\mathbb S_N^+)$ in Sec.\ \ref{sec:terracini locus spinor}. Finally, in Sec.\ \ref{sec:sing locus spinor} we give bounds on the singular locus $\Sing(\sigma_2(\mathbb S_N^+))$ and we conjecture what it actually we expect it to be. The appendix contains arguments suggesting that our conjecture should be true.

\section*{Notation}

For any integers $a\leq b \leq k \in \mathbb Z_{>0}$ we fix the notation 
\[ [k]:=\{1,\ldots, k\} \ \ \ \ , \ \ \ \ \binom{[k]}{a}:=\left\{ I \subset [k] \ | \ |I|=a\right\} \ , \]
and $i=a:b$ meaning $i\in \{a,a+1,\ldots, b\}$. Given $(e_1,\ldots, e_N)$ the standard basis of $\mathbb C^N$, we denote by $\bold{e}_I:=e_{i_1}\wedge \ldots \wedge e_{i_k} \in \bigwedge^k\mathbb C^N$ the exterior product of the basis vectors indexed by the subset $I \in \binom{[N]}{k}$.

\section{Preliminaries}\label{sec:preliminaries}

\indent \indent We refer to \cite[Ch.\ 5, Secc.\ 4-5]{procesi} for details on Clifford algebras and spin groups, and \cite[Ch.\ 11, Sec.\ 7]{procesi} for spin representations. Since Spinor varieties of Dynkin type $B_N$ and $D_{N+1}$ are projectively equivalent, they share the same geometry inside their minimal homogeneous embedddings. In this respect, we only work with spin groups of Dynkin type $D$ defined starting from even-dimensional vector spaces.

\paragraph{Clifford algebras and spin groups.} Given a finite-dimensional vector space $V$ over $\mathbb C$ endowed with a non-degenerate quadratic form $q\in \Sym^2V^\vee$, we denote by $\Cl_q(V)$ the {\em Clifford algebra} of $q$ over $V$ and by $\Cl_q^+(V)$ its even-degree component. One can always reduce to study Clifford algebras over even-dimensional vector spaces, thus we consider $\dim V=2N$.\\
\indent We fix the hyperbolic standard basis $(e_1,\ldots , e_N,f_1,\ldots , f_N)$ of $V$ with respect to the quadratic form $q=\sum_{k=1}^N x_kx_{N+k}$, so that the orthogonal decomposition $V=E\oplus E^\vee$ holds for $E=\langle e_1 , \ldots , e_N\rangle_\mathbb C$ and $E^\vee=\langle f_1,\ldots , f_N\rangle_\mathbb C$ fully isotropic subspaces of maximal dimension $\lfloor \frac{2N}{2} \rfloor=N$. Since $\bigwedge^\bullet V \simeq \Cl_q(V)$ (as vector spaces) \cite[Sec.\ 1.1, Proposition 1.3]{lawson}, the {\em Clifford multiplication} in $\Cl_q(V)$ can be described as follows: for any $v=e+f \in E\oplus E^\vee$ and $x \in \Cl_q(V)\simeq \bigwedge^\bullet V$ it holds 
\begin{equation}\label{multiplication}
	v\cdot x \simeq e \wedge x + f\neg x \ ,
	\end{equation}
where, for $x=x_1\wedge \ldots \wedge x_k\in \bigwedge^kV$,  $f \neg x:= \sum_{i=1}^k(-1)^{i+1}q(f,x_i)x_1\wedge \ldots \wedge \widehat{x_i}\wedge \ldots \wedge x_k \in \bigwedge^{k-1}V$. \\
\indent For any invertible element $x=v_1\cdots v_k \in \Cl_q(V)^\times$, its {\em spinor norm} and its {\em reverse involution} are respectively 
\begin{equation}\label{reverse in spin}
N(x):=xx^*=q(v_1)\cdots q(v_k) \in \mathbb C^\times \ \ \ , \ \ \ x^*:= v_k\cdots v_1 \ .
\end{equation}
Then the {\em spin group} is $\Spin(V,q):= \left\{x \in \Cl_q^+(V) ^\times \ | \ xVx^{-1}=V , \ N(x)=1\right\}$. For any $x\in \Spin(V,q)$ it holds $\big(\conj_x: v \mapsto xvx^{-1}\big)\in \SO(V,q)$. We write $\Spin_{2N}$ for the spin group $\Spin(\mathbb C^{2N},q)$ of Dynkin type $D_{N}$.

\paragraph{Spin representations.} In the same notation as above, for $V=E\oplus E^\vee$ of dimension $\dim V=2N$, the {\em half-spin representations} $V_{\omega_{N-1}}^{D_N}$ and $V_{\omega_{N}}^{D_N}$ for $\Spin_{2N}$ are defined by the two fundamental weights $\omega_{N-1}$ and $\omega_N$ on the right-hand-side extremal nodes of the Dynkin diagram of type $D_{N}$: 
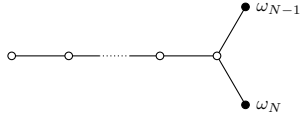
\[
\dynkin[labels={,,,,\omega_{N-1},\omega_{N}},edge length=.75cm] D{oo.oo**}
\]
They coincide with the $2^{N-1}$--dimensional vector spaces $\bigwedge^{{ev}}E$ and $\bigwedge^{{od}}E$ depending on the parity of $N$. More precisely, given $v_{\omega_i}$ and $\ell_{\omega_i}$ the highest and lowest weight vectors of the irreducible representation $V_{\omega_i}^{D_N}$ respectively, it holds:
\begin{table}[H]
	\centering
	\begin{tabular}{ |P{3cm}||P{3cm}|P{4cm}|P{2cm}|  }
		\hline
		\multicolumn{4}{|c|}{$N\equiv 0$ (mod $2$)} \\
		\hline
		\hline
		& $V_{\omega_i}^{D_N}$ & $v_{\omega_i}$ & $\ell_{\omega_i}$ \\
		\hline
		\hline
		$\omega_{N-1}$ & $\bigwedge^{od}E$ & $\bold{e}_{[N-1]}=e_1\wedge \ldots \wedge e_{N-1}$ & $e_1$ \\
		$\omega_N$ & $\bigwedge^{ev}E$ & $\bold{e}_{[N]}=e_1\wedge \ldots \wedge e_N$ & $\mathbbm{1}$ \\
		\hline
	\end{tabular}
	\caption{Half-spin representations of $\Spin_{2N}$ for $N\equiv 0$ (mod $2$).}
	\label{table:half-spin N even}
\end{table}
\vspace{-4mm}
\begin{table}[H]
	\centering
	\begin{tabular}{ |P{3cm}||P{3cm}|P{4cm}|P{2cm}|  }
		\hline
		\multicolumn{4}{|c|}{$N\equiv 1$ (mod $2$)} \\
		\hline
		\hline
		& $V_{\omega_i}^{D_N}$ & $v_{\omega_i}$ & $\ell_{\omega_i}$ \\
		\hline
		\hline
		$\omega_{N-1}$ & $\bigwedge^{ev}E$ & $\bold{e}_{[N-1]}=e_1\wedge \ldots \wedge e_{N-1}$ & $\mathbbm{1}$ \\
		$\omega_N$ & $\bigwedge^{od}E$ & $\bold{e}_{[N]}=e_1\wedge \ldots \wedge e_N$ & $e_1$ \\
		\hline
	\end{tabular}
	\caption{Half-spin representations of $\Spin_{2N}$ for $N\equiv 1$ (mod $2$).}
	\label{table:half-spin N odd}
\end{table}

\noindent where we write the scalar $\mathbbm{1}:=1\in \mathbb C$ when considered as weight vector (or simply as element of the representation). The half-spin representations $V_{\omega_{N-1}}^{D_N}$ and $V_{\omega_N}^{D_N}$ are self-dual if $N$ is even, and dual one to the other if $N$ is odd.

\paragraph{Spinor varieties.} The {\em Spinor varieties} $\mathbb S_N^+\subset \mathbb P(\bigwedge^{ev}E)$ and $\mathbb S_N^-\subset \mathbb P(\bigwedge^{od}E)$ are the rational homogeneous varieties $\Spin_{2N}\cdot [v_{\omega_i}]=\Spin_{2N}/P_{\omega_i}\subset\mathbb P(V_{\omega_i}^{D_N})$ for $P_{\omega_i}$ being the parabolic subgroup corresponding to the fundamental weight $\omega_i$, or equivalently the stabilizer of $v_{\omega_i}$. They have dimension $\binom{N}{2}$ and they are locally parametrized by the set of skew-symmetric matrices $\bigwedge^2\mathbb C^{N}$. Points in $V_{\omega_i}^{D_N}$ are said {\em spinors}, while the ones in $\mathbb S_N^\pm$ are said {\em pure spinors}. 

\begin{obs*}
	For $V=\mathbb Cu\oplus E\oplus E^\vee$ of odd dimension $2N+1$, the $B_N$--type spin group $\Spin_{2N+1}$ has only one {\em spin-representation} $V_{\omega_N}^{B_N}=\bigwedge^\bullet Eu$, hence only one Spinor variety $\mathbb S_{N}\subset \mathbb P(\bigwedge^\bullet Eu)$. 
	\end{obs*}

\hfill\break
\framebox[15cm]{
	\begin{minipage}{14.5cm}
		\centering
		\vspace{2mm}
		\begin{obs}[\textsc{Setting}]\label{rmk:setting} 
			The Spinor variety for $\Spin(2N+1)$ and the Spinor varieties for $\Spin(2N+2)$ are projectively equivalent. Then from now on we work only with $V=E\oplus E^\vee$ of dimension $\dim V =2N$ for an even $N$, hence with the Spinor variety $\mathbb S_{N}^+=\Spin_{2N}\cdot [\bold{e}_{[N]}] \subset \mathbb P\left( \bigwedge^{ev}E\right)$.
			All the arguments of this work can be repeated for $\dim V \equiv 2$ (mod $4$) with the Spinor variety $\mathbb S_{N}^-=\Spin_{2N}\cdot [\bold{e}_{[N]}] \subset \mathbb P(\bigwedge^{od}E)$.
		\end{obs}
	\vspace{0.5mm}
\end{minipage}}

\paragraph{Maximal orthogonal Grassmannians.} For any $a \in \bigwedge^\bullet E$, the Clifford multiplication \eqref{multiplication} defines a map 
\begin{equation}\label{psi_a}
	\begin{matrix}
		\psi_a: & V & \longrightarrow & \bigwedge^\bullet E\\
		& v & \mapsto & v\cdot a
		\end{matrix}
		\end{equation}
whose kernel $H_a:=\Ker(\psi_a)$ is always fully isotropic (i.e. $H_a\subset H_a^\perp$). Moreover, $H_a$ is of maximum dimension $\dim H_a=N$ (i.e. $H_a=H_a^\perp$) if and only if $[a]\in \mathbb S_{N}^+$. Notice that, in the notation from the previous paragraph, one has $H_{\bold{e}_{[N]}}=E$ and $H_{\mathbbm{1}}=E^\vee$.\\
\indent The set $\OG(N, V)$ of maximal fully isotropic subspaces of $V$ has two connected components, called {\em maximal orthogonal Grassmannians}:
\begin{align*}
\OG^+(N,V) &:= \left\{H \in \OG(N,V) \ | \ \codim_E(H\cap E)\equiv 0 \ (\text{mod} \ 2)\right\} \ , \\
\OG^-(N,V) &:=  \left\{H \in \OG(N,V) \ | \ \codim_E(H\cap E)\equiv 1 \ (\text{mod} \ 2)\right\} \ .
\end{align*}
Every $H \in \OG(N, V)$ defines a $1$--dimensional subspace $\left\{ \varphi \in \bigwedge^\bullet E \ | \ w \cdot \varphi=0 \ \forall w \in H\right\}=\langle a_H \rangle_\mathbb C \subset \bigwedge^\bullet E$, and $a_H$ is a pure spinor in either one of the two half-spin representations, depending on the parity of $N$. In particular, one has the one-to-one correspondence
	\begin{equation}\label{correspondence pure spinors and MFI subspaces}
	\begin{matrix}
	\mathbb S_{N}^+ \ \cup \  \mathbb S_{N}^- & \stackrel{1:1}{\longleftrightarrow} & \OG^+(N,V) \ \cup \ \OG^-(N,V)\\
	[a]  & \longrightarrow & H_a=\Ker(\psi_a) \\
	[a_H ] & \longleftarrow & H 
	\end{matrix} \ . \end{equation}	
\noindent Notice that every {\em non pure} spinor $a\in (\bigwedge^\bullet E)\setminus (\mathbb S_N^+\sqcup \mathbb S_N^-)$ defines a {\em non maximal} fully isotropic subspace $H_a=\Ker(\psi_a)$, but the latter does not define a unique spinor (cf. \cite[Sec.\ 3, Theorem 1]{batista2014pure}). We exhibit the association $H\mapsto [a_H]$. Given $H\in \OG^+(N,V)$ with $\dim(H \cap E)=p$ even, up to action of $\SO(V)$ one can assume $H=\langle e_1,\ldots , e_{p},g_{p+1},\ldots ,g_N\rangle_{\mathbb C}$ where 
\begin{equation}\label{generators for H_a}
	g_j= f_j+\sum_{k=p+1}^{N}\alpha_{kj}e_k \ \ \  ,\forall j=p+1:N \end{equation}
for a suitable skew-symmetric matrix $(\alpha_{kj})\in \bigwedge^2\mathbb C^{N-p}$. Then via \eqref{correspondence pure spinors and MFI subspaces} $H$ corresponds to a pure spinor $[a_H]\in \mathbb S_{N}^+$ such that $h\cdot a_H=0$ for any $h\in H$. The following result follows by a direct computation of the Clifford actions $g_j\cdot a_H=0$ for any $j=p+1:N$.

\begin{prop}\label{prop:from MFI to pure spinor}
	Let $N, p$ be even. Let $H=\langle e_1,\ldots , e_{p},g_{p+1},\ldots ,g_N\rangle_\mathbb C\in\OG^+(N,V)$, where the $g_j$'s are as in \eqref{generators for H_a} and define the skew-symmetric matrix $A=(\alpha_{kj})_{k,j}\in \bigwedge^{2}\mathbb C^{N-p}$. Then $H$ corresponds to the pure spinor 
	\[a_H=\sum_{I\subset\{p+1,\ldots, N\}}\Pf(A_I) \bold{e}_{[p]} \wedge \bold{e}_I \ , \] 
	where $\Pf(A_I)$ is the pfaffian of the submatrix $A_I$ of $A$ whose rows and columns are indexed by $I$, and we set $\bold{e}_\emptyset:=\mathbbm{1}$.
\end{prop}

\paragraph{Hamming distance of varieties.} Given $X\subset \mathbb P^M$ a projective variety, the {\em Hamming distance} $d(p,q)$ between two points $p,q\in X$, already considered in \cite{baur2007secant,abo2009non, galganostaffolani2022grass}, is the minimum number of lines of $\mathbb P^M$ lying in $X$ and connecting $p$ to $q$, that is
\begin{equation}\label{def:hamming distance}
	d(p,q):=\min \left\{ r \in \mathbb N \ | \ \exists \text{ distinct } p_1,\ldots, p_{r-1} \in X, \ L(p_i,p_{i+1})\subset X \ \forall i=0:r-1 \right\}
	\end{equation}
where we set $p_0=p$, $p_r=q$ and $L(p_i,p_{i+1})$ is the line passing through $p_i$ and $p_{i+1}$. In particular, $d(p,q)=0$ if and only if $p=q$, while $d(p,q)=1$ if and only if $L(p,q)\subset X$ and $p\neq q$. The {\em diameter} of $X$ is $\diam(X) := \max_{p,q \in X}\left\{d(p,q)\right\}$.

\begin{obs}\label{rmk:hamming for G-varieties}
	Given $G$ a group, $W$ a representation of $G$ and $X\subset \mathbb P(W)$ a variety which is invariant under the $G$-action induced on $\mathbb P(W)$, then the $G$-action preserves the Hamming distance between points in $X$, that is $d(g\cdot p, g\cdot q)=d(p,q)$.
	\end{obs}

For certain rational homogeneous varieties, called cominuscule, (among which, the Spinor varieties of $D$--type) the above distance coincides with the minimum possible degree of rational curves passing through the two points \cite[Lemma 4.2]{buch2013}. From the proof of \cite[Proposition 4.5]{buch2013} we get the following result.

\begin{prop}\label{prop:hamming distance}
	Let $[a],[b]\in \mathbb S_{N}^+$ be two pure spinors and let $H_a,H_b\in \OG^+(N,V)$ be their corresponding maximal fully isotropic subspaces. Then \[d([a],[b])=\frac{\codim_{H_a}(H_a \cap H_b)}{2} \ . \]
\end{prop}

\noindent As a consequence, for any $N \geq 2$ the Spinor varieties $\mathbb S_{N}^+$ have diameter $\left\lfloor \frac{N}{2} \right\rfloor$ (cf. table in \cite[Sec.\  4]{buch2013}). Finally, as Spinor varieties are intersection of quadrics, for any two distinct points $[a],[b]\in \mathbb S_{N}^+$ by Bèzout one gets $d([a],[b])\geq 2$ if and only if $L([a],[b])\cap \mathbb S_{N}^+=\{[a],[b]\}$.

\paragraph{Secant varieties.} The {\em secant variety of lines} $\sigma_2(\mathbb S_{N}^+) \subset \mathbb P(\bigwedge^{ev}E)$ of the Spinor variety $\mathbb S_{N}^+\subset \mathbb P(\bigwedge^{ev}E)$ is the union of all secant and tangent lines to $\mathbb S_{N}^+$ in $\mathbb P(\bigwedge^{ev}E)\simeq \mathbb P^{2^{N-1}-1}$. We denote by $\sigma_2^\circ(\mathbb S_N^+):=\{[a+b]\in \sigma_2(\mathbb S_N^+) \ | \ [a],[b]\in \mathbb S_N^+\}$ the set of all points in $\sigma_2(\mathbb S_N^+)$ lying on a bisecant line to $\mathbb S_N^+$. As a rational homogeneous variety, it is known that $\sigma_2(\mathbb S_{N}^+)=\overline{\Spin_{2N}\cdot [v_{\omega_i} + \ell_{\omega_i}]}$ is quasi-homogeneous \cite[Theorem 1.4]{zak1993tangents} and {\em non-defective} for any $N$ \cite{Kaj99,Ang10}, that is
\begin{equation}\label{dim of sec} 
\dim\sigma_2(\mathbb S_{N}^+) = \min\left\{2\dim\mathbb S_{N}^+ +1, \dim\mathbb P\left(\bigwedge^{ev}E\right) \right\} =\begin{cases}
2^{N-1} -1 & \text{for } N\leq 6\\
N(N-1)+1 & \text{for } N\geq 6
\end{cases} \ . 
\end{equation}
\noindent For $N \leq 5$ the secant variety $\sigma_2(\mathbb S_{N}^+)$ overfills the ambient space $\mathbb P^{2^{N-1}-1}$ (i.e. $\dim\sigma_2(\mathbb S_{N}^+)=2^{N-1}-1\lneq 2\dim \mathbb S_{N}^+ +1$), for $N \geq 7$ it is strictly contained, while for $N=6$ it perfectly fills $\mathbb P^{31}$ (in the sense $\dim\sigma_2(\mathbb S_{N}^+)=2\dim(\mathbb S_{N}^+)+1=2^{N-1}-1$): the latter case is said to be {\em perfect}.

\paragraph{Identifiability.} For any $[p] \in \sigma_2^\circ(\mathbb S_N^+)$, the {\em decomposition locus} of $p$ is the set of all pairs of pure spinors in $\mathbb S_N^+$ giving a decomposition of $[p]$, that is
\[ Dec_{\mathbb S_N^+}([p]) := \left\{([a],[b]) \ | \ [a],[b] \in \mathbb S_N^+,\ [p] \in \langle [a],[b]\rangle  \right\} \subset (\mathbb S_N^+)^2_{/\mathfrak S_2} \ , \]
where $(\mathbb S_N^+)^2_{/\mathfrak S_2}$ denotes the symmetric quotient of $(\mathbb S_N^+)^2$ by the symmetric group $\mathfrak{S}_2$. An element $([a],[b]) \in Dec_{\mathbb S_N^+}([p])$ is called a {\em decomposition} of $[p]$.\\
\indent A point $[p]\in \sigma_2^\circ(\mathbb S_N^+)$ is {\em identifiable} if $Dec_{\mathbb S_N^+}([p])$ is a singleton. Geometrically, this means that $[p]$ lies on a unique bisecant line to $\mathbb S_N^+$. Otherwise one says that $[p]$ is {\em unidentifiable}. For any subset $Y\subset \sigma_2^\circ(\mathbb S_N^+)$, we say that $Y$ is (un)identifiable if any point of $Y$ is so. In particular, an orbit is (un)identifiable if and only if a representative of it is so.

\paragraph*{Tangential-identifiability} The {\em tangential variety} of $\mathbb S_N^+$ in $\mathbb P(\bigwedge^{ev}E)$ is the union of all lines in $\mathbb P(\bigwedge^{ev}E)$ which are tangent to $\mathbb S_N^+$, i.e. $\tau(\mathbb S_N^+):=\bigcup_{p\in \mathbb S_N^+}T_p\mathbb S_N^+ \ \subset \mathbb P(\bigwedge^{ev}E)$. In particular, it holds $\mathbb S_N^+\subset \tau(\mathbb S_N^+)\subset \sigma_2(\mathbb S_N^+)$, and $\sigma_2(\mathbb S_N^+)=\sigma_2^\circ(\mathbb S_N^+)\cup \tau(\mathbb S_N^+)$ (the union may be non-disjoint). Moreover, from Zak's key result \cite[Theorem 1.4]{zak1993tangents} it holds that either $\dim\tau(\mathbb S_N^+) = 2\dim \mathbb S_N^+$ and $\dim \sigma_2(\mathbb S_N^+)=2\dim \mathbb S_N^+ +1$, or $\tau(\mathbb S_N^+)=\sigma_2(\mathbb S_N^+)$.

\begin{Def}\label{def:cactus}
	A tangent point $[q] \in \tau(\mathbb S_N^+)$ is {\em tangential-identifiable} if it lies on a unique tangent line to $\mathbb S_N^+$, or equivalently if there exists a unique $[p] \in \mathbb S_N^+$ such that $[q] \in T_{[p]}\mathbb S_N^+$. Otherwise it is {\em tangential-unidentifiable}.
\end{Def}

\noindent We say that an orbit is tangential-(un)identifiable if all of its elements are so.

\begin{Def}\label{def:tangential-locus}
	Given a tangent point $[q] \in \tau(\mathbb S_N^+)$, its {\em tangential-locus} is the set of points $p \in \mathbb S_N^+$ such that $q \in T_{p}\mathbb S_N^+$.
\end{Def}

\noindent Clearly, if $q$ is tangential-identifiable, then its tangential-locus is given by a single point at $\mathbb S_N^+$.

\paragraph{Nonabelian apolarity.} The classical apolarity action \cite{iarrobino1999power} and the skew-apolarity action \cite[ Definition 4]{ABMM21} are particular cases of a more general apolarity, namely the {\em nonabelian apolarity}, introduced at first in \cite[Section 1.3]{LO13}. Let $X \subset \mathbb P^n$ be a projective variety and let $\cal L$ be a very ample line bundle on $X$ giving the embedding $X \subset \mathbb P^n = \mathbb P( H^0(X,\cal L)^\vee)$. Let $\cal E$ be a rank--$e$ vector bundle on $X$ such that $H^0(X,\cal E^{\vee}\otimes \cal L)$ is not trivial, where $\cal E^\vee$ denotes the bundle dual to $\cal E$. The natural contraction map $H^0(X,\cal E) \otimes H^0(X,\cal E^{\vee}\otimes \cal L) \longrightarrow H^0(X,\cal L)$ leads to the morphism $A: H^0(X,\cal E) \otimes H^0(X,\cal L)^\vee \longrightarrow H^0(X,\cal E^{\vee} \otimes \cal L)^\vee$ and, after fixing a given $f \in H^0(X,\cal L)^\vee$, one gets the linear map
\[A_f : H^0(X,\cal E) \longrightarrow H^0(X,\cal E^{\vee}\otimes \cal L)^\vee \]
defined as $A_f(s) := A (s \otimes f)$ for any $s \in H^0(X,\cal E)$, and called {\em nonabelian apolarity action}. 
Let $H^0(X,\cal I_Z \otimes \cal E)$ and $H^0(X,\cal I_Z \otimes \cal E^{\vee} \otimes \cal L )$ be the spaces of global sections vanishing on a $0$--dimensional subscheme $Z\subset X$, and let $\langle Z\rangle\subset H^0(X,\cal L)^\vee$ be the linear span of $Z$. For a proof of the following result we refer to \cite[Proposition 5.4.1]{LO13}.

\begin{prop}\label{prop:nonabelian}
	Let $f\in H^0(X,\cal L)^\vee$ and let $Z \subset X$ be a $0$--dimensional subscheme such that $f \in \langle Z\rangle$. Then it holds $H^0(X,\cal I_Z \otimes \cal E) \subseteq \Ker A_f$ and $H^0(X,\cal I_Z \otimes \cal E^{\vee} \otimes \cal L ) \subseteq \im A_f^{\perp}$.
\end{prop}

\noindent It is worth remarking that the above result holds for non-reduced subschemes $Z\subset X$ too. The schematic non-reduced version of the nonabelian apolarity has already been considered in \cite[Theorem 6.10]{ottaviani2020tensor} for Veronese varieties, and more generally in \cite[Proposition 7]{galkazka2017vector}. In the classical apolarity, the case of minimal subschemes is considered with respect to cactus varieties and cactus rank \cite{cactusvarieties, cactusrank}, while for non-minimal subschemes see \cite{bernarditaufer}. Finally, for reduced subschemes we recall the following result \cite[Proposition 4.3]{oedingottaviani}.

\begin{prop}\label{prop:reduced-nonabelian}
	Let $f\in H^0(X,\cal L)^\vee$ and let $Z \subset X$ be a $0$--dimensional reduced subscheme of length $r$, minimal with respect to the property $f \in \langle Z\rangle$. Assume that $\rk A_f=r \cdot \rk \cal E$. Then it holds $H^0(X,\cal I_Z \otimes \cal E) = \Ker A_f$ and $H^0(X,\cal I_Z \otimes \cal E^{\vee} \otimes \cal L ) = \im A_f^{\perp}$. In particular, $Z$ is contained in the common zero locus of $\Ker A_f$ and $\im A_f^\perp$.
	\end{prop}

\section{Clifford apolarity}\label{sec:clifford apolarity}

\indent \indent In this section we analyze the nonabelian apolarity in the case of Spinor varieties, obtaining what we call {\em Clifford apolarity}. We also exhibit the vanishing condition of the global sections of a certain vector bundle on $\mathbb S_{N}^+$. We assume the setting in Remark \ref{rmk:setting}. \\
\hfill\break
\indent Consider the half-spin representations $V_{\omega_{N-1}^{D_N}}=\bigwedge^{od}E$ and $V_{\omega_{N}}^{D_N}=\bigwedge^{ev}E$, and consider the standard representation $V_{\omega_1}^{D_N}=V$. Let $\cal O(1)$ be the very ample line bundle $\cal O(1)$ on $\mathbb S_{N}^+$. Via the identification with $\OG^+(N,V)$, let $\cal U$ be the rank--$N$ universal bundle on $\mathbb S_{N}^+$ obtained as the pullback of the universal bundle on $\Gr(N,V)$. The bundles $\cal O(1)$, $\cal U^\vee$ and $\cal U(1):=\cal U\otimes \cal O(1)$ on $\mathbb S_{N}^+$ are homogeneous and the corresponding representations of the parabolic subgroup $P_{\omega_N}$ are irreducible with heighest weights $\omega_{N}$, $\omega_1$ and $\omega_{N-1}$ respectively: from Borel--Weil's Theorem one gets
\[ H^0\big(\mathbb S_{N}^+,\cal O(1)\big)\simeq \left( \bigwedge^{ev}E\right)^\vee \ \ \ , \ \ \ H^0\big(\mathbb S_{N}^+,\cal U(1)\big)\simeq \left( \bigwedge^{od}E\right)^\vee \ \ \ \, \ \ \ H^0\big(\mathbb S_{N}^+,\cal U^\vee\big)\simeq V^\vee \ . \]

\noindent The contraction map $H^0(\mathbb S_{N}^+, \cal U^\vee) \otimes H^0(\mathbb S_{N}^+, \cal U(1))  \rightarrow  H^0(\mathbb S_{N}^+, \cal O(1))$ is equivalent to
\[ \left(V_{\omega_1}^{D_N}\right)^\vee \otimes \left( V_{\omega_{N-1}}^{D_N}\right)^\vee \longrightarrow \left( V_{\omega_N}^{D_N} \right)^\vee \ ,  \]
and, from the duality of half-spin representations (as $N$ is even), we get
\begin{equation}\label{contraction b/w spin reps} 
	\left(E \oplus E^\vee\right) \otimes \bigwedge^{od}E  \longrightarrow  \bigwedge^{ev}E \ .
\end{equation}
\noindent Such map is uniquely determined as $\Spin(V)$-equivariant morphism (up to scalars): this is a consequence of Schur's lemma applied to the following result. 

\begin{teo}\label{thm:uniqueness of contraction map for spin reps}
	The irreducible $\Spin(V)$-module $\bigwedge^{ev}E$ (resp. $\bigwedge^{od}E$) appears with multiplicity $1$ in the $\Spin(V)$-module $V \otimes \bigwedge^{od}E$ (resp. $V \otimes \bigwedge^{ev}E$).
\end{teo}
\begin{proof}
	In light of the natural inclusion $\SL(E)\subset \Spin(V)$, for any $\Spin(V)$-module $W$ among the above ones we can consider its restriction as $\SL(E)$-module $\Res_{\SL(E)}^{\Spin(V)}(W)$: we lighten up the notation by simply writing $W$ and specifying its module structure. 
	Since $N$ is even, one can rewrite the $\SL(E)$--module $(E\oplus E^\vee)\otimes \bigwedge^{od}E$ as
	\[ (E\oplus E^\vee) \otimes \bigwedge^{od}E =  \bigoplus_{k=0}^{\frac{N-2}{2}} \left(E \otimes \bigwedge^{2k+1}E\right) \oplus \left(E^\vee \otimes \bigwedge^{2k+1}E\right) \ . \]
	\noindent From Pieri's formula \cite[Sec.\ 9.10.2]{procesi}, $(E\oplus E^\vee) \otimes \bigwedge^{od}E$ decomposes into the $\SL(E)$-modules
	\[V \otimes \bigwedge^{od}E = \mathbb C \oplus 2\left(\bigoplus_{k=1}^{\frac{N-4}{2}}\bigwedge^{2k}E \right) \oplus \bigwedge^{N}E \oplus \left( \bigoplus_{k=0}^{\frac{N-2}{2}} S^{\scriptsize \overbrace{(2,1,\ldots, 1)}^{2k+1}} \oplus  S^{\scriptsize(\overbrace{2, \ldots, 2}^{2k+1}, \overbrace{1,\ldots,1}^{N-2-2k})}\right) \ , \]
	among which there is only one copy of $\bigwedge^{ev}E$. Thus there has to be only one copy in the $\Spin(V)$-module decomposition as well.
\end{proof}

\indent By fixing $q \in \bigwedge^{ev}E$, from the map \eqref{contraction b/w spin reps} we get a $\Spin(V)$-equivariant map $\psi_q: (E \oplus E^\vee) \rightarrow  \bigwedge^{od}E$ which is unique (up to scalars) by Theorem \ref{thm:uniqueness of contraction map for spin reps}, thus it coincides with the map $\psi_q:V\rightarrow \bigwedge E$ defined in \eqref{psi_a}. Moreover, the natural contraction \eqref{contraction b/w spin reps} is equivalent to a map 
\[ \Phi: \bigwedge^{ev}E \otimes \bigwedge^{od}E^\vee \longrightarrow E\oplus E^\vee \]
which again by Theorem \ref{thm:uniqueness of contraction map for spin reps} is uniquely determined as $\Spin(V)$-equivariant morphism: more precisely, it is the projection onto the unique copy of the irreducible $\Spin(V)$-submodule $E\oplus E^\vee$.  

\begin{obs}\label{rmk:from skew to Clifford apolarity}
	By uniqueness, the map $\Phi$ is intrinsically related to the skew-apolarity map defined in \cite[Definition 4]{ABMM21}. Consider the splitting in $\SL(E)$-modules 
	\[ \bigwedge^{ev}E\otimes \bigwedge^{od}E^\vee= \bigoplus_{r=0}^{\frac{N}{2}}\bigoplus_{s=0}^{\frac{N-2}{2}}\left(\bigwedge^{2r}E \otimes \bigwedge^{2s+1}E^\vee\right) \ .\]
	The map $\Phi$ is $\Spin(V)$-equivariant, hence $\SL(E)$-equivariant, so are its restrictions
	\[ \Phi_{2r,2s+1} : \bigwedge^{2r}E\otimes \bigwedge^{2s+1}E^\vee \longrightarrow E\oplus E^\vee \ . \]
	By Schur's lemma, each restriction $\Phi_{2r,2s+1}$ is non-zero if and only if either $E$ or $E^\vee$ appears as irreducible $\SL(E)$-summand in the tensor product on the left-hand side: this happens if and only if $|2r-(2s+1)|=1$. In particular, it holds $\im\left(\Phi_{2r,2s+1}\right)=E$ if $2r= (2s+1)+1$, $\im\left(\Phi_{2r,2s+1}\right)=E^\vee$ if $2r= (2s+1)-1$, and $\im\left(\Phi_{2r,2s+1}\right)=0$ otherwise.\\
	For $2r=(2s+1)+1$ (resp. $2r=(2s+1)-1$), the module  $\bigwedge^{2r}E\otimes \bigwedge^{2s+1}E^\vee$ has a unique copy of $E$ (resp. $E^\vee$) as irreducible $\SL(E)$-submodule, thus (up to scalars) there exists a unique $\SL(E)$-equivariant morphism from the tensor product onto $E$ (resp. $E^\vee$). By uniqueness, the restrictions $\Phi_{2r,2s+1}$ are given (up to composing with a projection $\pi_{E\oplus E^\vee}$ onto $E\oplus E^\vee$) by generalizing the skew-catalecticant maps $\cal C^{s,d-s}_t$ in \cite[Definition 4]{ABMM21} as follows:
	\[
	\begin{matrix}
	\Phi_{2r,2s+1}: & \bigwedge^{2r}E\otimes \bigwedge^{2s+1}E^\vee& \rightarrow & \bigwedge^{2r-(2s+1)}E \oplus \bigwedge^{2s+1-2r}E^\vee & \stackrel{\pi_{E\oplus E^\vee}}{\twoheadrightarrow} & E\oplus E^\vee\\
	& e \otimes f & \mapsto &  \cal C^{2s+1, 2r-(2s+1)}_e(f) + \cal C^{2r,2s+1-2r}_f(e) \end{matrix}  \]
	where $\cal C_t^{s,d-s}=0$ for $s>d$: roughly, depending on the sign of $2r-(2s+1)$, one looks at vectors in $E^\vee$ as derivations on $E$, or viceversa.
\end{obs}

\begin{Def}\label{def:clifford apolarity}
	The {\em Clifford apolarity action} is the $\Spin(E\oplus E^\vee)$-equivariant map 
	\begin{equation}\label{Clifford apolarity}
		\begin{matrix}
			\Phi: & \bigwedge^{ev}E \otimes \bigwedge^{od}E^\vee & \longrightarrow & E\oplus E^\vee\\
			& e \otimes f & \mapsto & C_e(f)_{|_{E}} + C_f(e)_{|_{E^\vee}}
		\end{matrix}
	\end{equation}
	where $C_e(f)_{|_{E}}$ (resp. $C_f(e)_{|_{E^\vee}}$) is the projection onto $E$ (resp. $E^\vee$) of the contraction $f \cdot e\in \bigwedge E$ (resp. $e\cdot f\in \bigwedge E^\vee$) obtained via Clifford multiplication.
\end{Def}

\paragraph*{Vanishing of sections in $H^0(\mathbb S_{N}^+,\cal U(1))$.} Via the isomorphism $H^0(\mathbb S_{N}^+,\cal U(1))\simeq (V_{\omega_{N-1}}^{D_N})^\vee= \bigwedge^{od}E^\vee$, any section $j_f\in H^0(\mathbb S_{N}^+,\cal U(1))$ corresponds to a spinor $f \in \bigwedge^{od}E^\vee$: we describe the zero locus of a section $j_f$ by generalizing an argument from \cite[proof of Prop. 2]{manivel2020double}.\\
\indent For any $[a] \in \mathbb S_{N}^+$, one has the fiber $(\cal U(1))_{[a]} \simeq \Hom((\cal O(1))^\vee_{[a]},\cal U_{[a]})$. In particular, it holds $(\cal O(1))^\vee_{[a]}=\langle a \rangle_\mathbb C\subset \bigwedge^{ev}E$ and $\cal U_{[a]}=H_a:=\{v \in E\oplus E^\vee \ | \ v \cdot a =0\}$. Thus we can identify $j_f([a])\in (\cal U(1))_{[a]}$ with a homomorphism $j_f([a]): \langle a \rangle_\mathbb C \rightarrow H_a$. We want to determine $j_f([a])(\lambda a)\in H_a$ as the scalar $\lambda \in \mathbb C$ varies.

\begin{obs}\label{rmk on clifford scalar product}
	The scalar product $q(\cdot ,\cdot )$ on $V=E\oplus E^\vee$ extends to a scalar product $\langle \cdot , \cdot \rangle$ on $\bigwedge V$ defined on decomposable elements $x=x_1\wedge \ldots \wedge x_k \in \bigwedge^kV$ and $y=y_1\wedge \ldots \wedge y_h\in \bigwedge^hV$ as $\langle x,y\rangle =
	\det(q( x_i,y_j))$ if $k=h$, and zero otherwise. The isomorphism of vector spaces $\bigwedge V\simeq \Cl(V,q)$ defines such scalar product on the Clifford algebra too. The well-known adjointness between exterior product and contraction $\langle v \wedge x,y\rangle=\langle x,v \neg y\rangle$ for $v\in V$, $x \in \bigwedge^{h-1}V$ and $y \in \bigwedge^{h}V$, together with the Clifford multiplication \eqref{multiplication}, implies $\langle x,v\cdot y\rangle = \langle v\cdot x,y\rangle$ for any $x,y \in \bigwedge V$ and any $v \in V$: in the Clifford algebra, this property extends to any element $z \in \Cl(V,q)$ as $\langle x,zy\rangle = \langle z^*x,y\rangle$, where $z^*$ is as in \eqref{reverse in spin}.
\end{obs}

\indent The function $\psi_{f,a}(v)=\langle f,v\cdot a\rangle$ is linear in $v\in V$ and vanishes on $H_a$, thus it belongs to $H_a^\perp=\{\psi: V \rightarrow\mathbb C \ | \ \psi(H_a)=0 \}$. But $H_a=H_a^\perp$ (since $a$ is a pure spinor), thus $\psi_{f,\bullet}\in \Hom(\langle a \rangle_\mathbb C, H_a)$. It follows that we have two maps
\[ \begin{matrix} 
& \bigwedge^{od}E^\vee & \longrightarrow & H^0\left(\mathbb S_{N}^+,\cal U(1)\right)\\
j: & f & \mapsto & j_f\\
\Psi: & f & \mapsto & \psi_{f,\bullet}
\end{matrix}\]
where $j$ is the $\Spin(V)$-equivariant isomorphism given by Borel--Weil Theorem. Next remark shows that the maps $j$ and $\Psi$ coincide up to scalars.

\begin{obs}
	The map $\Psi$ is $\Spin(V)$-equivariant too. Indeed, for any $G$-homogeneous bundle on a variety $G/P$, the $G$-action on a global section $s$ is given by $(g \cdot s)(a)=s(g^{-1}\cdot a)$. Since for any $x\in \Spin(V)$ it holds $x^{-1}=x^*$, one gets $\left(x \cdot \psi_{f,\bullet}\right)(a)=\psi_{f,\bullet}(x^*a)=\psi_{f, x^*a}\in \Hom\left(\langle x^*a\rangle, H_{x^*a}\right)$ where $\psi_{f,x^*a}(\lambda x^*a)\in H_{x*a}^\perp=H_{x^*a}$ is the functional on $V$ given by $\psi_{f,x^*a}(\lambda x^*a)(v)=\langle f, v \cdot \lambda x^*a\rangle$. But $x^*\in \Spin(V)$ acts by conjugacy on $V$, thus $V=x^*Vx$ and by adjointness one concludes $\langle f, v\cdot \lambda x^* a\rangle = \langle f, (x^*wx)\cdot \lambda x^*a\rangle = \langle f, x^* w \cdot \lambda a\rangle \stackrel{adj.}{=} \langle x\cdot f, w \cdot \lambda a \rangle$, that is $x \cdot \psi_{f,\bullet}=\psi_{x\cdot f,\bullet}$.
\end{obs}

\begin{lemma}\label{lemma:vanishing of global sections}
	In the previous notation, the zero locus $\cal Z(j_f)$ of a global section $j_f\in H^0(\mathbb S_{N}^+,\cal U(1))$ corresponding to a spinor $f \in \bigwedge^{od}E^\vee$ is
	\begin{equation}\label{vanishing condition for sections}
		\cal Z(j_f)=\mathbb S_{N}^+ \cap (V\cdot f)^\perp \ .
	\end{equation}
\end{lemma}
\begin{proof}
	By Schur's lemma, the equivariant maps $j$ and $\Psi$ coincide up to scalar, hence the section $j_f\in H^0(\mathbb S_{N}^+,\cal U(1))$ is such that $j_f([a])(\lambda a)= \langle f,\bullet\cdot (\lambda a)\rangle \in H_a$ for any $\lambda \in \mathbb C$. We conclude that the zero locus $\cal Z(j_f)$ of the global section $j_f$ is given by the pure spinors $[a] \in \mathbb S_{N}^+$ such that $0= \langle f, v \cdot a \rangle = \langle v\cdot f , a \rangle$ for any $v \in V$, that is the thesis.
\end{proof}

\section{The poset of $\Spin_{2N}$--orbits in $\sigma_2(\mathbb S_N^+)$}\label{sec:poset spinor}

\indent \indent We assume the setting in Remark \ref{rmk:setting}. The secant variety of lines $\sigma_2(\mathbb S^+_N)$ is invariant under $\Spin_{2N}$--action, as well as its subsets $\sigma_2^\circ(\mathbb S_N^+)$ and $\tau(\mathbb S_N^+)$. We refer to the orbits lying in $\sigma_2^\circ(\mathbb S_{N}^+)$ as {\em secant orbits}, and to the orbits lying in $\tau(\mathbb S_{N}^+)$ as {\em tangent orbits}. Recall that the Spinor variety has diameter $\diam(\mathbb S_{N}^+)=\frac{N}{2}$. Moreover, for any $l=1:\left(\frac{N}{2}-1\right)$ we consider the following pure spinors and their corresponding maximal fully isotropic subspaces
\begin{equation}\label{def:representatives}
\begin{matrix}
\bold{e}_{[N-2l]} \ = \ e_1\wedge \ldots \wedge e_{N-2l} \ \in \mathbb S_{N}^+ \ , \\
\\
E_{N-2l}:= H_{\bold{e}_{[N-2l]}}=\langle e_1,\ldots, e_{N-2l}, f_{N-2l+1},\ldots , f_N\rangle_\mathbb C \ \in \OG^+(N,V) \ .
\end{matrix}
\end{equation}
By convention, for $l=\frac{N}{2}$ we set $\bold{e}_{[0]}=\ell_{\omega_N}=\mathbbm{1}$. On the other hand, for $l=0$ one gets $\bold{e}_{[N]}=v_{\omega_N}$.

\paragraph*{$Spin_{2N}$--orbits in $\mathbb S_N^+\times \mathbb S_N^+$.} The action of $\Spin_{2N}$ on $\mathbb S_{N}^+$ naturally induces the action on the direct product $\mathbb S_{N}^+\times \mathbb S_{N}^+$ given by $g\cdot ([a],[b])=([g\cdot a],[g\cdot b])$. We recall that the notation $g\cdot a$ stands for the action of $\Spin^q(V)\subset (\Cl_q(V))^\times$ on $\bigwedge E$ via Clifford multiplication \eqref{multiplication}.

\begin{obs}\label{rmk:spin on OG}
	Via the bijection \eqref{correspondence pure spinors and MFI subspaces}, the action on $\mathbb S_{N}^+\times \mathbb S_{N}^+$ is equivalent to the action of $\Spin^q(V)$ on $\OG^+(N,V)^{\times 2}$ given by $g \cdot (H_a,H_b)=(g\cdot H_a,g\cdot H_b)$, where $g \cdot H_a=gH_ag^{-1}$ is the {\em conjugacy} action: this follows from the inclusion $\Spin^q(V)\subset N_{\Cl_q(V)^\times}(V)$.
\end{obs}

\indent From Remark \ref{rmk:hamming for G-varieties}, the Hamming distance in $\mathbb S_{N}^+$ is invariant under $\Spin^q(V)$-action, as well as the dimensions of subspaces in $V$ are preserved under conjugacy. In particular, the actions of $\Spin^q(V)$ on $\mathbb S_{N}^+\times \mathbb S_{N}^+$ and on $\OG^+(N,V)^{\times 2}$ restrict to actions on the subsets
\begin{equation}\label{orbit O_l} 
\cal O_{l,N} := \left\{ ([a],[b]) \in \mathbb S_{N}^+\times \mathbb S_{N}^+ \ | \ d([a],[b])=l\right\}
\end{equation}
\begin{equation}\label{orbit in OGxOG}
\widehat{\cal O}_{l,N} := \left\{ (H_a,H_b) \in \OG^+(N,V)^{\times 2} \ | \ \dim(H_a \cap H_b)=N-2l \right\}
\end{equation}
respectively, for any $l=0:\frac{N}{2}$ where $\frac{N}{2}=\diam(\mathbb S_N^+)$. Since for $l=0$ one gets the diagonals $\Delta_{\mathbb S_N^+}$ and $\Delta_{\OG^+(N,V)}$ we consider $l\geq 1$. From Proposition \ref{prop:hamming distance}, for any distance $l\geq 1$, the subsets $\cal O_{l,N}$ and $\widehat{\cal O}_{l,N}$ are equivalent one to the other, and they give partitions of the corresponding direct products (as sets). Since $\dim (H_{\bold{e}_{[N]}}\cap H_{\mathbbm{e}_l})=\dim (E\cap E_l)=N-2l$, for any $l=1:\frac{N}{2}$ one gets the inclusion $\Spin^q(V)\cdot \left([\bold{e}_{[N]}],[\bold{e}_{[N-2l]}]\right) \subset \cal O_{l,N}$.

\begin{obs*}
	It is likely well-known that the $\Spin_{2N}$--orbit partition of $\Q^{2N-2}\times \Q^{2N-2}$ is uniquely determined by the Hamming distance: in particular, given $\Delta_{\Q^{2N-2}}$ the diagonal, it holds
	\begin{equation}\label{partition QxQ}
	\Q^{2N-2}\times \Q^{2N-2} \ = \ \Delta_{\Q^{2N-2}} \ \sqcup \ \Spin_{2N}\cdot ([e_1],[e_2]) \ \sqcup \ \Spin_{2N}\cdot ([e_1],[f_1]) \ . 
	\end{equation}
\end{obs*}

\begin{prop}\label{prop:O_l is orbit}
	For any $l=1:\diam(\mathbb S_{N}^+)$, the spin group $\Spin^q(V)$ acts transitively on $\cal O_{l,N}\subset \mathbb S_{N}^+\times \mathbb S_{N}^+$. In particular, it holds $\cal O_{l,N}= \Spin^q(V)\cdot ([\bold{e}_{[N]}],[\bold{e}_{[N-2l]}])$.
\end{prop}

\begin{proof}
	Given $([a],[b])\in \cal O_{l,N}$, by homogeneity of $\mathbb S_N^+$ we may assume $a=\bold{e}_{[N]}$.\\
	\indent First we prove the result for $l=\frac{N}{2}$. Since $d([\bold{e}_{[N]}],[b])=\frac{N}{2}$, it holds $E \cap H_{b}=\{0\}$ and we may assume $H_b=\langle g_1,\ldots , g_N\rangle_\mathbb C$ for generators $g_j$ as in \eqref{generators for H_a}. In light of \eqref{partition QxQ} $\Spin^q(V)$ conjugates the fully isotropic vector $g_1=f_1+\sum_{i=2}^N\alpha_{i1}e_i$ (having Hamming distance $2$ from $e_{1}$) to $f_1$ by leaving $e_1$ fixed. Now, consider $V'=\langle e_2,\ldots, e_N,f_2,\ldots, f_N\rangle_\mathbb C$ and the subspaces $E'=E\cap V'$ and $H_b'=H_b/\langle e_1,f_1\rangle_\mathbb C$: again $\Spin(V')$ conjugates $g_2-\alpha_{12}e_1 \in H_b'$ to $f_2$ by leaving $e_2$ fixed. In particular, $\Spin^q(V)$ conjugates $g_2\in H_b$ to $f_2$ by leaving $e_1,f_1,e_2$ fixed. By iterating, $\Spin^q(V)$ conjugates $H_b$ to $E^\vee$, hence $([\bold{e}_{[N]}],[b])$ to $(\bold{e}_{[N]},[\mathbbm{1}])$.\\
	\indent On the other hand, for $l\lneq \frac{N}{2}$ one has $H_b=\langle h_1,\ldots, h_{N-2l}, g_{N-2l+1},\ldots, g_N\rangle_\mathbb C$ for $g_j$'s as in \eqref{generators for H_a} and $E\cap H_b=\langle h_1,\ldots, h_{N-2l}\rangle_\mathbb C$. Up to reordering $h_1,\ldots ,h_{N-2l}$, one gets that $h_j$ has Hamming distance $1$ from $e_j$ for any $j=1:N-2l$, hence from \eqref{partition QxQ} and similar arguments as above one can conjugates $E\cap H_b$ to $E_{N-2l}$. Finally, by working in $W:=V/E_{N-2l}$ one can conjugates $E/E_{N-2l}$ to $H_b/E_{N-2l}$ via $\Spin(W)$ and lifting this to a conjugation under $\Spin^q(V)$ leaving $E_{N-2l}$ fixed. The thesis follows.
\end{proof}

\begin{cor}\label{cor:orbit partition of SxS}
	The product $\mathbb S_{N}^+\times \mathbb S_{N}^+$ splits in the $\Spin_{2N}$--orbits
	\[ \mathbb S_N^+ \times \mathbb S_N^+ = \Delta_{\mathbb S_N^+} \sqcup  \bigsqcup_{l=1}^{ \frac{N}{2} } \Spin_{2N}\cdot \left( [\bold{e}_{[N]}],[\bold{e}_{[N-2l]}]\right) \ . \]
\end{cor}

\subsection{Secant orbits in $\sigma_2^\circ(\mathbb S_{N}^+)$}\label{subsec:secant orbits}

\indent \indent We deduce the orbit partition of the dense subset $\sigma_2^\circ(\mathbb S_{N}^+)$ from Proposition \ref{prop:O_l is orbit}. We recall that each spinor $[a+b]\in \mathbb P\left(\bigwedge^{ev}E\right)$ defines, via the map \eqref{psi_a}, a fully isotropic subspace $H_{a+b}=\Ker(\psi_{a+b})$, which has maximal dimension $N$ if and only if $[a+b]$ is pure. By definition of these subspaces as annihilators, it clearly holds $H_a \cap H_b \subseteq H_{a+b}$.\\
\indent For any two distinct pure spinors $[a],[b]\in \mathbb S_{N}^+$ such that $d([a],[b])=1$, the spinor $[a+b]$ is pure too, since the line $L([a],[b])$ fully lies in the Spinor variety: in particular, in this case one has $\dim(H_a\cap H_b)\stackrel{\eqref{prop:hamming distance}}{=}N-2$ while $\dim H_{a+b}=N$, thus the strict inclusion $H_a\cap H_b \subsetneq H_{a+b}$ holds. However, for higher Hamming distances the equality holds. 
\begin{lemma}\label{lemma:dim H_{w+e_l}}
	In the above notation, it holds $H_{\bold{e}_{[N]} + \bold{e}_{[N-2l]}}=E \cap E_{N-2l}$ if and only if $l\neq 1$. In particular, for any $l\geq 2$, it holds $\dim H_{\bold{e}_{[N]} + \bold{e}_{[N-2l]}}=N-2l$.
\end{lemma}
\begin{proof}
	We already know that $\langle e_1 , \ldots , e_{N-2l}\rangle_\mathbb C=E\cap E_{N-2l} \subset H_{\bold{e}_{[N]} + \bold{e}_{[N-2l]}}$ holds. Assume that there exists $v \in H_{\bold{e}_{[N]} + \bold{e}_{[N-2l]}}$ such that $v \notin (E\cap E_{N-2l}) \setminus \{0\}$. Thus, since $H_{\bold{e}_{[N]} + \bold{e}_{[N-2l]}}$ is fully isotropic, up to linear combinations we can consider a decomposition of $v$ with respect to the standard hyperbolic basis $(e_1,\ldots, e_N,f_1,\ldots, f_N)$ of the form $v= \alpha_1 e_{N-2l+1} + \ldots + \alpha_{2l}e_N + \beta_1f_{N-2l+1} + \ldots + \beta_{2l}f_N$ for some $\alpha_i,\beta_j \in \mathbb C$. Since $H_{\bold{e}_{[N]} + \bold{e}_{[N-2l]}}=\Ker(\psi_{\bold{e}_{[N]} + \bold{e}_{[N-2l]}})$, it holds 
	\begin{align*}
	0 & = v \cdot (\bold{e}_{[N]} + \bold{e}_{[N-2l]})\\
	& \stackrel{\eqref{multiplication}}{=} \sum_{i=1}^{2l}\alpha_i(e_1\wedge \ldots \wedge e_{N-2l}\wedge e_{N-2l+i})+ \sum_{j=1}^{2l}(-1)^{j+1}\beta_j(e_1\wedge \ldots \wedge \hat{e}_{N-2l+j} \wedge \ldots \wedge e_N) \ ,
	\end{align*}
	where $\hat{e}_{N-2l+j}$ denotes that the vector is missing in the wedge product. \\
	Now, the summands above may simplify one to each other if and only if $l= 1$, otherwise they do not since they are all independent vectors in the standard basis of $\bigwedge E$. It follows that, for $l\geq 2$, it holds $v\cdot (\bold{e}_{[N]} + \bold{e}_{[N-2l]})=0$ if and only if $\alpha_i=\beta_j=0$ for all $i,j\in [2l]$, that is $H_{\bold{e}_{[N]} + \bold{e}_{[N-2l]}}=E \cap E_l$. On the other hand, for $l=1$, the conditions $\alpha_1=\beta_2$ and $\alpha_2=-\beta_1$ give a non-zero vector $0\neq v\in H_{\bold{e}_{[N]} + \bold{e}_{[N-2l]}}\setminus (E\cap E_l)$, thus $E \cap E_1 \subsetneq H_{\bold{e}_{[N]} + \bold{e}_{[N-2l]}}=\langle e_1,\ldots , e_{N-2}, e_{N-1}+ f_N, e_N - f_{N-1}\rangle_\mathbb C$.
\end{proof}

\begin{cor}\label{cor:fiber of Sigma_l}
	For any $[a+b]\in \sigma_2^\circ(\mathbb S_{N}^+) \setminus \mathbb S_{N}^+$, the equality $H_{a+b}=H_a\cap H_b$ holds.
\end{cor}
\begin{proof}
	Let $[a+b]=[c+d] \in \sigma^\circ(\mathbb S_{N}^+) \setminus \mathbb S_{N}^+$ be such that $d([a],[b])=l$ and $d([c],[d])=m$ for certain $2\leq l,m\leq \frac{N}{2}$. By Proposition \ref{prop:O_l is orbit} there exists $g \in \Spin^q(V)$ such that $g\cdot ([a],[b])= ([\bold{e}_{[N]}],[\bold{e}_{[N-2l]}])$ and $g\cdot ([c],[d])= ([c'],[d'])$. In particular, $[\bold{e}_{[N]} + \bold{e}_{[N-2l]}]=g\cdot [a+b]=g\cdot [c+d] = [c'+d']$, and by Proposition \ref{lemma:dim H_{w+e_l}} we get $H_{c'}\cap H_{d'}\subset H_{c'+d'}=H_{\bold{e}_{[N]} + \bold{e}_{[N-2l]}}=E\cap E_{N-2l}$, where the last equality follows from Lemma \ref{lemma:dim H_{w+e_l}} since $l\geq 2$. Dimensionally we have $N-2m=\dim(H_{c'}\cap H_{d'})\leq \dim(E \cap E_{N-2l})=N-2l$. But, by symmetry, one also gets $N-2l=\dim(H_{a'}\cap H_{b'})\leq \dim(E \cap E_{N-2m})=N-2m$, thus $l=m$ and the thesis follows.
\end{proof}

\indent Corollary \ref{cor:fiber of Sigma_l} allows to define for any $l=2:\frac{N}{2}$ the subset
\begin{align}\label{def:secant orbits}
\Sigma_{l,N}& :=\left\{ [a+b]\in \sigma_2^\circ(\mathbb S_{N}^+) \ | \ \dim H_{a+b}=N-2l\right\} \nonumber\\
& \ = \left\{[a+b]\in \sigma_2^\circ(\mathbb S_{N}^+) \ | \ d([a],[b])=l\right\} \ .
\end{align}
Moreover, we set $\Sigma_{1,N}:=\mathbb S_{N}^+$. Our claim is that the subsets $\Sigma_{l,N}$ are exactly the $\Spin^q(V)$-orbits in $\sigma_2^\circ(\mathbb S_{N}^+)$.\\
\indent The action of $\Spin^q(V)$ on $\sigma_2(\mathbb S_{N}^+)$ preserves the subsets $\Sigma_{l,N}$, as by Remark \ref{rmk:spin on OG} the spin group acts on $V$ and its subspaces by conjugacy. Moreover, by Proposition \ref{prop:O_l is orbit} any two pairs $([a],[b])$ and $([c],[d])$ of Hamming distance $l$ are conjugated, hence their lines $L([a],[b])$ and $L([c],[d])$ are so. Finally, the following result proves that $\Spin^q(V)$ acts transitively on points on a same line $L([a],[b])\setminus\{[a],[b]\}$ too.

\begin{lemma}\label{lemma:transitivity on lines}
	For any two distinct pure spinors $[a],[b]\in \mathbb S_{N}^+$, the spin group $\Spin^q(V)$ acts transitively on $L([a],[b])\setminus\{[a],[b]\}$.
\end{lemma}
\begin{proof}
	Since the lines defined by pairs of pure spinors having the same Hamming distance are all conjugated, it is enough to prove the transitivity on the line $L([\bold{e}_{[N]}],[\bold{e}_{[N-2l]}])$. Moreover, given a point $[\lambda \bold{e}_{[N]} + \mu \bold{e}_{[N-2l]}]=[\bold{e}_{[N]} + z \bold{e}_{[N-2l]}] \in L([\bold{e}_{[N]}],[\bold{e}_{[N-2l]}])$, we can rewrite it as $\bold{e}_{[N]}+z\bold{e}_{[N-2l]}=e_1\wedge \ldots \wedge e_{N-2l}\wedge \left(e_{N-2l+1}\wedge \ldots \wedge e_N + z\mathbbm{1}\right)$. Since $[e_{N-2l+1}\wedge \ldots \wedge e_N]$ and $[z\mathbbm{1}]$ are pure spinors in $\mathbb S_{2l}^+$, we can restrict to consider the line $L([\bold{e}_{[N]}],[\mathbbm{1}])$. \\
	\indent Given $\bold{e}_{[N]}+z \mathbbm{1}$, we look for a spin element $g\in\Spin^q(V)$ such that $\bold{e}_{[N]}+ z \mathbbm{1}= k (\bold{e}_{[N]} + \mathbbm{1})$ for some $k \in \mathbb C^\times$. We consider the element $\tilde g=(a_1e_1+b_1f_1)\cdots (a_Ne_N+b_Nf_N)\in \Cl_q^+(V)$ for certain $a_i,b_i\in \mathbb C^\times$: it is product of an even number of vectors in $V$ by the assumption as $N$ is even. Then $\tilde g\in\Spin^q(V)$ if and only if it is invertible and it has unitary spinor norm, that is if and only if $\prod_{i=1}^N a_ib_i\stackrel{\clubsuit}{=}1$. Via Clifford multiplication \eqref{multiplication} it holds  $(a_1e_1 + b_1f_1)\cdot (\bold{e}_{[N]} + z \mathbbm{1})=a_1ze_1+ b_1e_2\wedge \ldots \wedge e_N$,
	and by iterating for $i=1:N$ one gets $\tilde g \cdot (\bold{e}_{[N]} + z\mathbbm{1})= a_1\cdots a_Nz \bold{e}_{[N]} + b_1\cdots b_N \mathbbm{1}$. In particular, the condition $\tilde g \cdot (\bold{e}_{[N]}+z\mathbbm{1})=k(\bold{e}_{[N]}+\mathbbm{1})$ is equivalent to the condition $z\prod_{i=1}^N a_i\stackrel{\spadesuit}{=}\prod_{i=1}^N b_i$. By putting together the conditions $(\clubsuit)$ and $(\spadesuit)$, it is straightforward that for the choice $a_1=\ldots =a_N=\sqrt[2N]{z^{-1}}$ and $b_1=\ldots=b_N=\sqrt[2N]{z}$ one gets $\tilde g\cdot (\bold{e}_{[N]} + z \mathbbm{1})=\sqrt{z}(\bold{e}_{[N]} + \mathbbm{1})$.
\end{proof}

It follows that for any $l=1:\frac{N}{2}$ the subset $\Sigma_{l,N}$ is a $\Spin^q(V)$-orbit in $\sigma_2^\circ(\mathbb S_{N}^+)$. Moreover, by Proposition \ref{prop:O_l is orbit} we deduce that for any $l$ it holds $\Sigma_{l,N}=\Spin(V)\cdot [\bold{e}_{[N]}+ \bold{e}_{[N-2l]}]$.\\
In conclusion, we have proved the following theorem.

\begin{teo}\label{thm:secant orbit in sec2}
	The dense set $\sigma_2^\circ(\mathbb S_{N}^+)$ splits under the action of $\Spin_{2N}$ in the orbits
	\[ \sigma_2^\circ(\mathbb S_{N}^+)= \bigsqcup _{l=1}^{\frac{N}{2}} \Spin_{2N}\cdot [\bold{e}_{[N]}+ \bold{e}_{[N-2l]}] \ .\]
\end{teo}

\subsection{Tangent orbits in $\tau(\mathbb S_{N}^+)$}\label{subsec:tangent orbits}

\indent \indent From the non-defectivity of $\sigma_2(\mathbb S_{N}^+)$ and the dicotomy between tangential and secant varieties \cite[Corollary $4$]{FH79}, we know that $\tau(\mathbb S_4^+)=\sigma_2(\mathbb S_4^+)=\mathbb P^{7}$ and $\tau(\mathbb S_5^+)=\sigma_2(\mathbb S_5^+)=\mathbb P^{15}$, while for $N\geq 6$ the tangential variety $\tau(\mathbb S_{N}^+)$ is a divisor in $\sigma_2(\mathbb S_{N}^+)\subset \mathbb P^{2^{N-2}-1}$. We deduce the orbit partition of $\tau(\mathbb S_{N}^+)$ from the tangent bundle on the Spinor variety.\\
\indent Let $\cal T_{\mathbb S_N^+}$ be the tangent bundle on the Spinor variety $\mathbb S_{N}^+$. From the parametrization of Spinor varieties, one can describe the fiber at any pure spinor $[a]\in \mathbb S_N^+$ as
\[ \left(\cal T_{\mathbb S_{N}^+} \right)_{[a]} \ = \ T_{[a]}\mathbb S_{N}^+ \ \simeq \ \bigwedge^2H_a  \ , \]
leading to the isomorphism of homogeneous bundles $\cal T_{\mathbb S_{N}^+} \ \simeq \  \bigwedge^2\cal U^\vee$, where $\cal U$ is the rank--$N$ universal bundle on $\mathbb S_N^+$ obtained by pulling back the universal bundle on the Grassmannian $\Gr(N,V)$. \\
\indent By homogeneity of $\mathbb S_{N}^+$, all tangent spaces are conjugated one to each other by transformations in $\Spin_{2N}\setminus P_{\omega_{N}}$. Thus the $\Spin_{2N}$--orbits of points in the tangential variety $\tau(\mathbb S_{N}^+)$ are in bijection with the $P_{\omega_N}$--orbits in the tangent space $T_{[v_{\omega_N}]}\mathbb S_N^+\simeq \bigwedge^2E$, which are parametrized by the possible ranks of skew-symmetric matrices of size $N$. Let $\left[ T_{[a]}\mathbb S_{N}^+\right]_{2l}$ be the set of tangent points to $\mathbb S_{N}^+$ at the pure spinor $[a]$ corresponding to skew-symmetric matrices of size $N$ and rank $2l$. Then for any $l=1:\frac{N}{2}$ we denote the set of all tangent points corresponding to rank--$2l$ skew-symmetric matrices by
\begin{equation}\label{def:tangent orbits}
\Theta_{l,N}:=\left\{[q]\in \tau(\mathbb S_{N}^+) \ | \ \exists [a]\in \mathbb S_{N}^+ \ : \ [q]\in \left[T_{[a]}\mathbb S_{N}^+ \right]_{2l}\right\} \ .
\end{equation}
The above arguments ensure that each subset $\Theta_{l,N}$ is indeed a $\Spin^q(V)$-orbit, and all together they give the $\Spin^q(V)$-orbit partition of $\tau(\mathbb S_{N}^+)$.\\
\indent Finally, for any $l=1:\frac{N}{2}$ the tangent orbit $\Theta_{l,N}$ admits as representative the spinor
\begin{equation}\label{def:tangent representative}
[q_l]:=\left[ \sum_{i=1}^{l}e_{2i-1}\wedge e_{2i} \right] \ .
\end{equation}
\noindent Indeed, the curve of rank--$2l$ skew-symmetric matrices of size $N$
\begin{equation}\label{standard skewsymm matrix}
tC_l=t{\tiny \begin{bmatrix}
	P_1 \\
	& \ddots \\
	& & P_l \\
	& & & 0_{N-2l} 
	\end{bmatrix}} \ \ \ \text{where} \ \ \ P_i={\scriptsize \begin{bmatrix}
	0 & 1 \\ -1 & 0
	\end{bmatrix}} 
\end{equation} 
defines the curve of maximal fully isotropic subspaces
\[ H_{c(t)}=\left\langle f_1+t\sum_{k=1}^{N}c_{k1}e_k, \ldots, f_N+t\sum_{k=1}^{N}c_{kN}e_k \right\rangle_\mathbb C \ \in \ \OG^+(N,V) \ , \]
which by Proposition \ref{prop:from MFI to pure spinor} corresponds to the curve of pure spinors $c(t)= \sum_{I \in 2^{[N]}}\Pf(C(t)_I)\bold{e}_I$ passing at $c(0)=\mathbbm{1}$ with direction $c'(0)=q_l$, thus $[q_l]\in T_{[\mathbbm{1}]}\mathbb S_{N}^+$. This proves the next theorem.

\begin{teo}\label{thm:tangent orbit in sec2}
	The tangential variety $\tau(\mathbb S_{N}^+)$ splits in the $\Spin_{2N}$-orbits
	\[ \tau(\mathbb S_{N}^+) \ = \ \bigsqcup_{l=1}^{\frac{N}{2}}\Spin_{2N}\cdot \left[\sum_{i=1}^l e_{2i-1}\wedge e_{2i}\right] \ .\]
\end{teo}

\subsection{Inclusions among closures of $\Spin_{2N}$--orbits}

\indent \indent We have treated the secant orbits and the tangent orbits separately, now we analyze the behaviour of their inclusions. First of all, notice that the tangent representative $[q_1]=[e_1\wedge e_2]\in\Theta_{1,N}$ is a pure spinor, hence 
\[ \Theta_{1,N}=\mathbb S_{N}^+=\Sigma_{1,N} \ . \] 

\begin{obs*} 
	The orbit $\Theta_{1,N}$ is given by skew-symmetric matrices having rank (as matrices) $2$, hence it is described by the Grassmannian of planes $\Gr(2,N)$. This agrees with the more theoretical result \cite[Prop. 2.5 \& Subsec. 3.1]{LM03} $T_x\mathbb S_{N}^+ \cap \mathbb S_{N}^+\simeq\Gr(2,N)$ for any $x \in \mathbb S_{N}^+$.
\end{obs*}

Moreover, for $l=2$, a representative of $\Theta_{2,N}$ is $[q_2]=[e_1\wedge e_2 + e_3\wedge e_4]$. But $[e_1\wedge e_2]$ and $[e_3\wedge e_4]$ are pure spinors with corresponding subspaces $H_{e_1\wedge e_2}=E_{[2]}$ and $H_{e_3\wedge e_4}=\left\langle f_1,f_2,e_3,e_4,f_5, \ldots , f_N \right\rangle_\mathbb C$, hence they have Hamming distance $2$ and $[q_2]\in \Sigma_{2,N}$. Thus 
\[ \Theta_{2,N} = \Sigma_{2,N} \ . \]

\begin{obs}\label{rmk:subspaces for tangent points}
	Via the map $\psi_q$ in \eqref{psi_a} every $[q]\in \Theta_{l,N}$ defines a subspace $H_q\in \OG(N-2l,V)$: indeed, the representative $[q_l]$ \eqref{def:tangent representative} defines the subspace $H_{q_l}=\langle f_{2l+1},\ldots, f_N\rangle_{\mathbb C}$ of dimension $N-2l$ and $\Spin^q(V)$ acts by conjugacy on the subspaces, preserving their dimensions. Moreover, by definition, a tangent point $[q]\in T_{[a]}\mathbb S_{N}^+\cap\Theta_{l,N}$ is the direction of a curve $\gamma(t)=\{[a(t)] \ | \ t \in (-\epsilon, \epsilon)\}\subset\mathbb S_{N}^+$ passing at $\gamma(0)=[a]$ and, up to considering a smaller neighbourhood, one can assume $d([a],[a(t)])=m$ for any $t\in (0,\epsilon)$. Then for any $t> 0$ the spinor $\frac{a-a(t)}{t}$ defines a subspace $H_t\in \OG(N-2l,V)$ and one gets the equality $H_q=\lim_{t \rightarrow 0}H_t \in \OG(N-2l,V)$.
\end{obs}

\indent From the previous arguments we deduce the following description of the secant variety of lines
\[ \sigma_2(\mathbb S_{N}^+)=\mathbb S_{N}^+ \sqcup \left( \bigsqcup_{l=2}^{\frac{N}{2}} \Sigma_{l,N}\right)\cup \left( \bigsqcup_{l=3}^{\frac{N}{2}} \Theta_{l,N} \right) \ ,\]
where the non-disjoint unions appears since we haven't proved $\Sigma_{l,N}\neq \Theta_{l,N}$ for $l\geq 3$ yet. Since we are interested in considering $l\geq 3$, we assume $N\geq 6$.

\begin{lemma}\label{lemma:orbits inclusions}\*
	\begin{enumerate}
		\item For any $l=2:\frac{N}{2}$ it holds $\Sigma_{l-1,N} \subset \overline{\Sigma_{l,N}}$. 
		\item For any $l=2:\frac{N}{2}$ it hols $\Theta_{l-1,N}\subset \overline{\Theta_{l,N}}$.
		\item For any $l=3:\frac{N}{2}$ it holds $\Theta_{l,N} \subset \overline{\Sigma_{l,N}}$.
	\end{enumerate}
\end{lemma}
\begin{proof}
	\begin{enumerate}
		\item For $\epsilon >0$ consider the sequence $[\bold{e}_{[N]}+a_\epsilon] \in \Sigma_{l,N}$ for $[a_\epsilon]\in \mathbb S_N^+$ defined by the subspaces $H_{a_\epsilon}= \left\langle e_1,\ldots, e_{N-2l}, g_{N-2l+1}(\epsilon), g_{N-2l+2}(\epsilon), g_{N-2l+3}, \ldots , g_{N} \right\rangle_\mathbb C\in \OG^+(N,V)$ where 
		\[ g_{N-2l+1}(\epsilon)=\frac{1}{\epsilon}f_{N-2l+1}+  e_{N-2l+2} \ \ \ , \ \ \ g_{N-2l+2}(\epsilon)=\frac{1}{\epsilon}f_{N-2l+2}-  e_{N-2l+1} \]
		and $g_h=f_h+\sum_{k=N-2l+3}^N\alpha_{kj}e_k$ as in \eqref{generators for H_a}. Then the sequence $[\bold{e}_{[N]}+a_\epsilon]$ has limit $[\bold{e}_{[N]}+a]$ where the pure spinor $[a]\in \mathbb S_{N}^+$ corresponds to the maximal fully isotropic subspace $H_a= \left\langle e_1,\ldots, e_{N-2l+2}, g_{N-2l+3}, \ldots , g_{N} \right\rangle_\mathbb C$: in particular, $[\bold{e}_{[N]}+a]\in \Sigma_{l-1,N}$. By reversing this argument, one can always look at a point in $\Sigma_{l-1,N}$ as limit of a sequence in $\Sigma_{l,N}$.
		\item The tangent points in $\Theta_{l-1,N}$ correspond to skew-symmetric matrices of size $N$ and rank $2l-2$, while points in $\Theta_{l,N}$ to rank--$2l$ skew-symmetric matrices of size $N$. As the former matrices lie in the closure of the latter ones, the thesis follows.
		\item Consider $[q]\in \Theta_{l,N} \subset \tau(\mathbb S_{N}^+)$: then by Remark \ref{rmk:subspaces for tangent points} $\dim H_{q}=N-2l$ and there exists a curve of pure spinors $\gamma(t)=\{ [a(t)] \ | \ t \in (-\epsilon, \epsilon)\}\subset \mathbb S_{N}^+$ with direction $\gamma'(0)=[q]$ such that (up to a smaller $\epsilon$) $d([a(t)],[a(0)])=l$ for any $t\in (0,\epsilon)$. In particular, for any $t>0$ it holds $\left[\frac{a(t)+a(0)}{t}\right]\in \Sigma_{l,N}$ and by definition $[q]=\left[\lim_{t\rightarrow 0}\frac{a(t)-a(0)}{t}\right]=\lim_{t\rightarrow 0}\left[\frac{a(t)+a(0)}{t}\right]\in \overline{\Sigma_{l,N}}$.
	\end{enumerate}
\end{proof}

\begin{lemma}\label{lemma:Tau_l neq Sigma_l}
	For any $N\geq 6$ and $l=3:\frac{N}{2}$, it holds $\Sigma_{l,N}\neq \Theta_{l,N}$.
\end{lemma}
\begin{proof}
	For $l=\frac{N}{2}$ the equality does not hold since for $N\geq 6$ one has $\tau(\mathbb S_{N}^+)\subsetneq \sigma_2(\mathbb S_{N}^+)$. \\
	By contradiction, assume that there exists $3\leq l \leq \frac{N}{2}-1$ such that the equality holds. Then Lemma \ref{lemma:orbits inclusions} implies $\Sigma_{i,N} = \Theta_{i,N}$ for any $3 \leq i \leq l$: in particular, $\Sigma_{3,N} = \Theta_{3,N}$. Since we want to deal with $3 \lneq \frac{N}{2}$, we assume $N\geq 7$. \\
	\indent Consider the tangent representative $[q_3] = [e_1\wedge e_2 + e_3\wedge e_4 + e_5\wedge e_6] \in \Theta_{3,N}$. It defines the subspace $H_{q_3}=\langle f_7,\ldots, f_N\rangle_{\mathbb C}$ of dimension $\dim H_{q_3}=N-6$. Set $E\cap E_{[6]}=\langle e_1,e_2,e_3,e_4,e_5,e_6\rangle_{\mathbb C}$. Since $[q_3]\in \Sigma_{3,N}$, there exist pure spinors $[a],[b] \in \mathbb S_{N}^+$ such that $[q_3]=[a+b]$ and $d([a],[b])=3$. Then, being kernels, one gets $\langle f_7,\ldots, f_N\rangle_{\mathbb C}=H_{q_3}=H_{a+b}=H_a\cap H_b$, where the last equality follows from Corollary \ref{cor:fiber of Sigma_l}. This implies that $a+b\in \bigwedge^{ev}(E\cap E_{[6]})$ and $[a],[b]\in \mathbb S_{6}^+$: in particular, the equality $[q_3]=[a+b]$ holds in $\Sigma_{3,6}$ for the representative $[q_3]\in \Theta_{3,6}$. But this means that $\Theta_{3,6}=\Sigma_{3,6}$, that is $\tau(\mathbb S_6^+)=\sigma_2(\mathbb S_{6}^+)$ which is a contradiction.
\end{proof}

\begin{teo}\label{thm:orbit partition of sec}
	For any $N\geq 6$, the poset of $\Spin_{2N}$--orbits in the secant variety of lines $\sigma_2(\mathbb S_{N}^+)$ is described by the graph in Figure \eqref{figure:graph}, where arrows denote the inclusion of an orbit into the closure of the other orbit. In particular, the orbits $\Theta_{\frac{N}{2},N}$ and $\Sigma_{\frac{N}{2},N}$ are the dense orbits of the tangential and secant variety respectively.
\end{teo}
\begin{figure}[H]
	\begin{center}
		{\small \begin{tikzpicture}[scale=2.3]
			
			\node(S) at (0,0.2){{$\mathbb S_{N}^+$}};
			\node(t2) at (0,0.6){{$\Theta_{2,N}=\Sigma_{2,N}$}};
			\node(t3) at (-0.6,1){{$\Theta_{3,N}$}};
			\node(t) at (-0.6,1.5){{$\vdots$}};
			\node(td) at (-0.6,2){$\Theta_{\frac{N}{2},N}$};
			
			\node(s3) at (0.6,1.2){{$\Sigma_{3,N}$}};
			\node(s) at (0.6,1.7){{$\vdots$}};
			\node(sd) at (0.6,2.2){{$\Sigma_{\frac{N}{2},N}$}};
			
			\path[font=\scriptsize,>= angle 90]
			(S) edge [->] node [left] {} (t2)
			(t2) edge [->] node [left] {} (t3)
			(t3) edge [->] node [left] {} (s3)
			(t3) edge [->] node [left] {} (t)
			(t) edge [->] node [left] {} (td)
			(t) edge [->] node [left] {} (s)
			(td) edge [->] node [left] {} (sd)
			(s3) edge [->] node [left] {} (s)
			(s) edge [->] node [left] {} (sd);
			\end{tikzpicture}}
	\end{center}
	\caption{Poset graph of $\Spin_{2N}$--orbits in $\sigma_2(\mathbb S_{N}^+)$.}
	\label{figure:graph}
\end{figure}
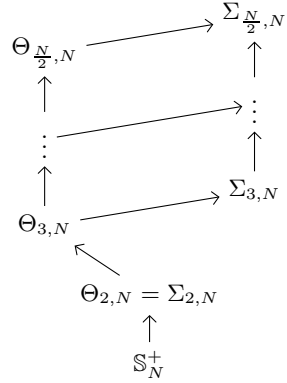

\section{Identifiability in $\sigma_2(\mathbb S_{N}^+)$}\label{sec:identifiability spinor}

\indent \indent From Sec.\ \ref{sec:preliminaries} we recall the preliminary notions of (un)identifiability and decomposition locus. We assume the setting in Remark \ref{rmk:setting}. 

\paragraph*{Unidentifiability and decomposition loci for $\Sigma_{2,N}$} The $6$--dimensional Spinor variety $\mathbb S_{4}^+\subset \mathbb P^7$ coincides with the $6$-dimensional quadric $\Q^6\subset \mathbb P^7$. In this case, the secant orbit $\Sigma_{2,4}$ is dense in $\sigma_2(\mathbb S_4^+)=\mathbb P^7$. Given $[q]\in \mathbb P^7\setminus \Q^6$, for any pair $([x],[y])\in Dec([q])$ in the decomposition locus it holds $d([x],[y])=2$, since $L([x],[y])\nsubseteq\Q^6$.

\begin{prop}\label{prop:decomposition locus S4}
	For any $[q]\in \mathbb P^7\setminus \Q^6$, the decomposition locus $Dec([q])$ is parametrized by the lines $\mathbb P^1\subset \mathbb P^7$ passing at $[q]$ and intersecting $\Q^6$ in two distinct points, i.e. 
	\[Dec([q])\simeq \mathbb P^6\setminus \Q^5 \ . \]
	In particular, the distance-$2$ orbit $\Sigma_{2,4}$ is unidentifiable.
\end{prop}
\begin{proof} 
	The set of lines $\mathbb P^1\subset \mathbb P^7$ passing through $[q]$ and intersecting $\Q^6$ in two distinct points is isomorphic to $\mathbb P^6\setminus \Q^5$: indeed, points in $\Q^5\subset \mathbb P^6$ correspond to lines $\mathbb P^1\subset \mathbb P^7$ which are tangent to $\Q^6$. A pair $([x],[y])\in (\Q^6)^{\times 2}$ gives a decomposition for $[q]=[x+y]$ if and only if the line $L([x],[y])\subset\mathbb P^7$ contains $[q]$ and intersects $\Q^6$ exactly in the distinct points $[x],[y]$. Thus the thesis follows.
\end{proof}

\begin{es}\label{unidentifiable in S_4}
	For $\mathbb S_4^+$, consider the representative $[\bold{e}_{[4]}+\mathbbm{1}]$ of the dense orbit $\Sigma_{2,4}$: an alternative decomposition is \[\left[\bold{e}_{[4]} + \mathbbm{1}\right] \ = \ [(\bold{e}_{[4]} - \bold{e}_{[2]})+ (\bold{e}_{[2]} + \mathbbm{1})] \]
	where both summands are pure spinors since $d([\bold{e}_{[4]}],[\bold{e}_{[2]}])=d([\bold{e}_{[2]}],[\mathbbm{1}])=1$.
\end{es}

\begin{prop}\label{prop:contact locus for Sigma2}
	For any $N\geq 6$, let $[q]\in \Sigma_{2,N}$ defining the $(N-4)$--dimensional isotropic subspace $H_q\subset \mathbb C^{2N}$ via $\psi_q$ in \eqref{psi_a}. Let $\Q^6\subset \mathbb P^7=\mathbb P(H_q^\perp/H_q)$ be the quadric in the (projectivization of the) orthogonal quotient of $H_q$. For any $[\tilde{a}]\in \Q^6$, we denote by $[a]=[\tilde{a}\wedge H_q]$ the pure spinor in $\mathbb S_N^+$. Then the decomposition locus of $[q]$ is $6$--dimensional and isomorphic to the open set
	\[Dec([q])\simeq \Q^6\setminus \left\{ [\tilde{a}] \in \Q^6 \ | \ L([a],[q]) \subset \mathbb S_N^+ \right\} \ . \]
	In particular, the distance-$2$ orbit $\Sigma_{2,N}$ is unidentifiable.
\end{prop}
\begin{proof}
	Fix $[q]\in \Sigma_{2,N}$. For any $([a],[b])\in Dec([q])$, the pure spinors $[a],[b]$ have Hamming distance $2$ and define two maximal fully isotropic subspaces $H_a,H_b$ such that $H_a\cap H_b=H_q$. In particular, $\dim (H_a\cap H_b)=N-4$ and $\dim \left(\frac{(H_a\cap H_b)^\perp}{H_a\cap H_b}\right)=8$. Thus, in the orthogonal quotient space $W:=H_q^\perp/H_q\simeq \mathbb C^8$, they give the $4$--dimensional (maximal) isotropic subspaces $H_a/H_q, H_b/H_q$ which intersect trivially, hence they correspond to two pure spinors $[\tilde a],[\tilde b]\in \mathbb S_4^+\simeq \Q^6$ such that $d([\tilde a],[\tilde b])=2$: in particular, the line $L([\tilde a],[\tilde b])\nsubseteq \Q^6$ as well as $L([a],[b])=L([a],[q])\nsubseteq \mathbb S_N^+$. \\
	\indent On the other hand, start from a pure spinor $[\tilde a]\in \Q^6$ (look at the quadric as the Spinor variety $\mathbb S_4^+$ constructed from $W$) and consider the lifting of its corresponding subspace $\tilde H_{\tilde a}\in \OG^+(4,W)$ to the subspace $H:=\langle \tilde H_{\tilde a},H_q\rangle_{\mathbb C}\in \OG^+(N,V)$, which corresponds to the pure spinor $[a]=[\tilde{a}\wedge H_q]\in \mathbb S_{N}^+$. If $[\tilde a]\in \Q^6$ is such that $L([a],[q])\nsubseteq \mathbb S_N^+$, then by Bézout ($\mathbb S_N^+$ is intersection of quadrics) there exists a unique $[b]\in \mathbb S_{N}^+$ such that $L([a],[q])\cap \mathbb S_{N}^+=\{[a],[b]\}$.\\
	We conclude that each $[\tilde a] \in \Q^6$ such that $L([a],[q])\nsubseteq \mathbb S_N^+$ corresponds to a unique pair $([a],[b])\in Dec([q])$ in the decomposition locus of $[q]$, hence the thesis.
\end{proof}

\paragraph*{Identifiability of orbits $\Sigma_{l,N}$ for $l\geq 3$.} We show the identifiability of the secant orbits $\Sigma_{l,N}$ for $l\geq 3$ via an inductive argument, based on the injectivity of a wedge-multiplication map. The {\em base case} of the induction is given by the identifiability of the dense orbit $\Sigma_{\frac{N}{2},N}$, which we prove via Clifford apolarity introduced in Sec. \ref{sec:clifford apolarity}.

\begin{obs*}
	The case of the $15$--dimensional Spinor variety $\mathbb S_6^+\subset \mathbb P^{31}$ was already known \cite[Example 5]{AR03} as an example of {\em variety with one apparent double point} (OADP variety), namely a $n$--dimensional variety $X\subset \mathbb P^{2n+1}$ such that through a general point of $\mathbb P^{2n+1}$ there passes a unique secant line to $X$ \cite[Definition 3]{AR03}.
\end{obs*}

\begin{lemma}\label{lemma:generic identifiability of sec2}
	For any $N\geq 6$, the dense secant orbit $\Sigma_{\frac{N}{2},N}$ is identifiable.
\end{lemma}
\begin{proof}
	We prove that the representative $[\bold{e}_{[N]}+\mathbbm{1}]\in \Sigma_{\frac{N}{2},N}$ is identifiable. In the notation of Sec. \ref{sec:clifford apolarity}, consider the map $\Phi_{\bold{e}_{[N]}+\mathbbm{1}}:H^0(\mathbb S_{N}^+,\cal U(1))\rightarrow H^0(\mathbb S_{N}^+,\cal U^\vee)^\vee$ corresponding to the Clifford apolarity \eqref{Clifford apolarity} 
	\[ \begin{matrix}
	\Phi_{\bold{e}_{[N]}+\mathbbm{1}}: &\bigwedge^{od}E^\vee & \rightarrow & E\oplus E^\vee\\
	& f & \mapsto & C_{\bold{e}_{[N]}+\mathbbm{1}}(f)_{|_E}+ C_{f}(\bold{e}_{[N]}+\mathbbm{1})_{|_{E^\vee}} 
	\end{matrix} \ . \]
	Since $\rk \Phi_{\bold{e}_{[N]}+\mathbbm{1}}=2N=2 \rk \cal U(1)$, from Proposition \ref{prop:reduced-nonabelian} it is enough to prove that the common zero locus of $\Ker(\Phi_{\bold{e}_{[N]}+\mathbbm{1}})$ is $\cal Z\left( \Ker(\Phi_{\bold{e}_{[N]}+\mathbbm{1}})\right)=\{[\bold{e}_{[N]}],[\mathbbm{1}]\}$. \\
	Given $f=\sum_{k=0}^{\frac{N-2}{2}}\sum_{I \in \binom{[N]}{2k+1}}c_I\bold{f}_I\in \bigwedge^{od}E^\vee$, one gets 
	\[\Phi_{\bold{e}_{[N]}+\mathbbm{1}}(f) = \big( f\cdot \bold{e}_{[N]} + f\cdot \mathbbm{1}\big)_{|_E} + \big( \bold{e}_{[N]}\cdot f + \mathbbm{1}\cdot f\big)_{|_{E^\vee}}= \big(f\cdot \bold{e}_{[N]}\big)_{|_E} + f_{|_{E^\vee}} \ , \]
	hence
	\[ \Ker(\Phi_{\bold{e}_{[N]}+\mathbbm{1}})= \bigwedge^{3}E^\vee \oplus \ldots \oplus \bigwedge^{N-3}E^\vee \ .\]
	\indent Consider $a=\sum_{s=0}^{\frac{N}{2}}\sum_{J \in \binom{[N]}{2s}}\alpha_{J}\bold{e}_J\in \bigwedge^{ev}E$ such that $[a]\in \cal Z\left(\Ker(\Phi_{\bold{e}_{[N]}+\mathbbm{1}})\right)$. From the vanishing of global sections in Lemma \ref{lemma:vanishing of global sections}, we know that $\langle V\cdot f,a\rangle=0$ for any $f \in \Ker(\Phi_{\bold{e}_{[N]}+\mathbbm{1}})$. In particular, for any index subset $I\subset [N]$ of odd cardinality from $3$ to $N-3$ (so that $\bold{f}_I$ lies in the kernel), and for any basis vector $e_\lambda \in E$, we get 
	\begin{align}\label{vanishing with e's}
	0 & = \langle e_\lambda \cdot \bold{f}_I, a\rangle = \langle \bold{f}_I, e_\lambda \wedge a \rangle  = \sum_{s=0}^{\frac{N}{2}}\sum_{J \in \binom{[N]}{2s}}\alpha_J\langle \bold{f}_I,e_\lambda\wedge \bold{e}_J\rangle \\
	& = \sum_{s=0}^{\frac{N}{2}} \sum_{J \in \binom{[N]}{2s} \ : \ J\not \owns \lambda} (-1)^{pos(\lambda,J\cup\{\lambda\})+1} \alpha_J\langle \bold{f}_I,\bold{e}_{J\cup\{\lambda\}}\rangle = (-1)^{pos(\lambda, I)+1} \alpha_{I\setminus \{\lambda\}}	 \ ,
	\end{align}
	\noindent where $pos(\lambda,I)$ denotes the position of $\lambda$ in the ordered subset $I$, and the last equality follows from Remark \ref{rmk on clifford scalar product}. Thus in $a$ there is no summand indexed by a subset $J$ such that $J\subset I$ and $|J|=|I|-1$ for any subset $I$ with $|I|=3,5, \ldots, N-3$. On the other hand, for any subset $I\subset [N]$ of odd cardinality from $3$ to $N-3$ (i.e. $\bold{f}_I\in \Ker(\Phi_{\bold{e}_{[N]}+\mathbbm{1}})$) and for any basis vector $f_\nu \in E^\vee$ such that $\nu \notin I$, it holds
	\begin{align}\label{vanishing with f's}
	0 & = \langle f_\nu \cdot \bold{f}_I, a\rangle = \langle f_\nu\wedge \bold{f}_I, a \rangle  = (-1)^{pos(\nu, I\cup \{\nu\})+1}\sum_{s=0}^{\frac{n+1}{2}}\sum_{J \in \binom{[N]}{2s}}\alpha_J\langle \bold{f}_{I\cup \{\nu\}},\bold{e}_J\rangle \\
	& = (-1)^{pos(\nu, I\cup \{\nu\})+1} \alpha_{I\cup \{\nu\}}	 \ ,
	\end{align}
	implying that $a$ has no summand indexed by a subset $J$ such that $4 \leq |J|\leq N-2$. We deduce that every $[a]\in \cal Z\left( \Ker(\Phi_{\bold{e}_{[N]}+\mathbbm{1}}) \right)$ is such that $a=\alpha + \beta \bold{e}_{[N]} \in \mathbb C \oplus \bigwedge^{N}E$. Since $d(\mathbbm{1},\bold{e}_{[N]})=\frac{N}{2}$, the line $L([\bold{e}_{[N]}],[\mathbbm{1}])$ does not lie in $\mathbb S_{N}^+$: as Spinor varieties are intersections of quadrics, by Bézout the only points of the form $[\alpha\mathbbm{1} + \beta\bold{e}_{[N]}]$ being pure spinors are for either $\alpha=0$ or $\beta=0$, that is $\cal Z\left( \Ker(\Phi_{\bold{e}_{[N]}+\mathbbm{1}})\right)=\{[\bold{e}_{[N]}],[\mathbbm{1}]\}$.
\end{proof}

\begin{teo}\label{thm:identifiable secant orbits}
	For any $N\geq 6$ and $l\geq 3$, the secant orbit $\Sigma_{l,N}\subset \sigma_2(\mathbb S_{N}^+)$ is identifiable.
\end{teo}
\begin{proof}
	Fix $l\geq 3$ and consider the orbit $\Sigma_{l,N}=\Spin_{2N}\cdot [\bold{e}_{[N]}+\bold{e}_{[N-2l]}]$. By homogeneity, it is enough to show that the spinor $[\bold{e}_{[N]} +\bold{e}_{[N-2l]}]$ is identifiable.\\
	\indent Assume {\em ad absurdum} that there exist two pure spinors $[a],[b]\in \mathbb S_{N}^+\setminus\{[\bold{e}_{[N]}],[\bold{e}_{[N-2l]}]\}$ such that $[a+b]=[\bold{e}_{[N]} + \bold{e}_{[N-2l]}]$. In particular, since $l\neq 1$, Corollary \ref{cor:fiber of Sigma_l} implies that $H_a\cap H_b=H_{a+b}=H_{\bold{e}_{[N]} + \bold{e}_{[N-2l]}}=E \cap E_{N-2l}$, where we recall that $E_{N-2l}=\langle e_1,\ldots ,e_{N-2l},f_{N-2l+1},\ldots , f_N\rangle_{\mathbb C}$. Since $\langle e_1, \ldots, e_{N-2l}\rangle_{\mathbb C}=E\cap E_{N-2l} \subset H_a, H_b$, by Proposition \ref{prop:from MFI to pure spinor} it follows that there exist two spinors $a',b' \in \bigwedge^{ev}\left(\langle e_{N-2l+1}, \ldots, e_N\rangle_{\mathbb C}\right)$ such that $a=\bold{e}_{[N-2l]}\wedge a'$ and $b=\bold{e}_{[N-2l]}\wedge b'$: by maximality of $H_a$ and $H_b$, the spinors $[a'],[b']\in \mathbb S_{2l}^+$ are pure in a smaller Spinor variety where they have maximum Hamming distance $l$. Moreover, since $[a],[b]\in \mathbb S_{N}^+\setminus \{[\bold{e}_{[N]}],[\bold{e}_{[N-2l]}]\}$, we also know that $[a'],[b']\in \mathbb S_{2l}^+\setminus \{[e_{N-2l+1}\wedge \ldots \wedge e_N],[\mathbbm{1}]\}$. \\
	The spinor $[a+b]=[\bold{e}_{[N-2l]}\wedge(a'+b')]\in \Sigma_{l,N}$ is the image of $[a'+b']\in \Sigma_{l,2l}$ via the wedge-multiplication map 
	\begin{equation}\label{wedge-multiplication} 
	\bold{e}_{[N-2l]}\wedge \bullet \ : \ \bigwedge^{ev}\langle e_{N-2l+1}, \ldots, e_N\rangle_{\mathbb C} \ \longrightarrow \bigwedge^{ev}E
	\end{equation}
	restricting to
	$(\bold{e}_{[N-2l]} \wedge \bullet):\Sigma_{l,2l}\rightarrow \Sigma_{l,N}$. Since the above linear map is injective, the equality $
	\big(\bold{e}_{[N-2l]}\wedge \bullet\big)([a'+b'])  = \big(\bold{e}_{[N-2l]}\wedge \bullet\big)([e_{N-2l+1}\wedge \ldots \wedge e_N + \mathbbm{1}])$ implies that $[a'+b']=[e_{N-2l+1}\wedge \ldots \wedge e_N + \mathbbm{1}]$. But the spinor $[e_{N-2l+1}\wedge \ldots \wedge e_N + \mathbbm{1}]$ is a representative for the dense orbit $\Sigma_{l,2l} \subset \sigma_2(\mathbb S_{2l}^+)$, which is identifiable by Lemma \ref{lemma:generic identifiability of sec2} (since $l\geq 3$). Thus it holds $\{[a'],[b']\}=\{[e_{N-2l+1}\wedge \ldots \wedge e_N],[\mathbbm{1}]\}$, in contradiction to $[a'],[b']\in \mathbb S_{2l}^+\setminus \{[e_{N-2l+1}\wedge \ldots \wedge e_N],[\mathbbm{1}]\}$.
\end{proof}

 \section{Tangential-identifiability in $\tau(\mathbb S_{N}^+)$}\label{sec:tangential-identifiability}

\indent \indent In this section we prove that any point of a tangent orbit $\Theta_{l,N}$ for $l\geq 3$ is tangential-identifiable (cf. Definition \ref{def:cactus}), and we do so via Clifford apolarity (cf. Section \ref{sec:clifford apolarity}). We assume the setting in Remark \ref{rmk:setting}. Since for $N\leq 5$ it holds $\overline{\Theta_{2,N}}=\tau(\mathbb S_{N}^+)=\sigma_2(\mathbb S_{N}^+)$, in the following we assume $N\geq 6$.

\begin{obs}\label{rmk:S_6 tangential-identifiability}
	The $15$--dimensional Spinor variety $\mathbb S_6^+\subset \mathbb P^{31}$ has been deeply studied in the context of the {\em Freudenthal's magic square} \cite{LM01}. Indeed, $\mathbb S_6^+$ is the third of the four {\em Legendrian varieties} lying in the third row of the square (associated to the group $\Spin_{12}$ - see \cite{LM01}), while $\mathbb P^{31}=\sigma_2(\mathbb S_6)^+$ is the third prehomogenous space in the fourth row (associated to the exceptional group $E_7$ - see \cite{Cle03}). In both the above references the generic tangential-identifiability of $\tau(\mathbb S_{6}^+)$ is obtained: we refer to \cite[Propp.\ 5.8-5.12]{LM01} and \cite[Propp.\ 8.4,9.8]{Cle03} for details.
\end{obs}

\indent In the notation of Section \ref{sec:clifford apolarity}, for a tangent spinor $[q]\in\tau(\mathbb S_{N}^+)\setminus \mathbb S_{N}^+$ we consider the map
\[ \Phi_{q} : H^0\left(\mathbb S_{N}^+,\cal U(1)\right) \longrightarrow H^0\left(\mathbb S_{N}^+,\cal U^\vee\right)^\vee \]
corresponding to the Clifford apolarity in \eqref{Clifford apolarity}. Similarly to the proof of Lemma \ref{lemma:generic identifiability of sec2}, we apply the nonabelian apolarity, but in this case we deal with non-reduced subschemes. Let $Y \subset \mathbb S_{N}^+$ be a non-reduced subscheme of $\mathbb S_N^+$ of length $2$ such that $[q]\in \langle Y\rangle$, and let $Y_{supp}$ be its support: in particular, such a $Y$ corresponds to $\{[p],[q]\}$ for $[p]\in \mathbb  S_{N}^+$ such that $[q] \in T_{[p]}\mathbb S_{N}^+$, and $Y_{supp}=\{[p]\}$. Then by Proposition \ref{prop:nonabelian} we get $H^0\left(\mathbb S_{N}^+, I_Y\otimes \cal U(1)\right)\subset\Ker(\Phi_q)$ and in particular $\cal Z(\Ker(\Phi_q))\subset Y_{supp}$, where $\cal Z(\Ker(\Phi_q))$ is the common zero locus of global sections in $\Ker(\Phi_q)$. If such common zero locus is given by only one pure spinor $[p]\in \mathbb S_{N}^+$, then for any $0$-dimensional subscheme $Y\subset X$ of length $2$ such that $[q]\in \langle Y\rangle$ it holds $Y_{supp}=\{[p]\}$, that is $[q]$ lies on only one tangent space, namely $T_{[p]}\mathbb S_{N}^+$, and $[q]$ is tangential-identifiable.

\begin{teo}\label{thm:tangential-identifiable tangent orbit}
	For any $N\geq 6$ and $l\geq 3$, the tangent orbit $\Theta_{l,N}\subset \tau(\mathbb S_{N}^+)$ is tangential-identifiable.
\end{teo}
\begin{proof} 
	Fix $l=3:\frac{N}{2}$. From the above argument it is enough to prove that $\cal Z\left(\Ker(\Phi_{q_l}) \right)=\{[\mathbbm{1}]\}$ for $[q_l]\in T_{[\mathbbm{1}]}\mathbb S_{N}^+\setminus \mathbb S_{n+1}^+$ being the representative of $\Theta_{l,N}$ as in \eqref{def:tangent representative}: 
	\[ q_l=\sum_{i=1}^l e_{2i-1}\wedge e_{2i} \ \in \bigwedge^2E \ . \]
	\noindent From the Clifford apolarity, for any $f \in \bigwedge^{od}E^\vee$ we have $\Phi_{q_l}(f)=C_{q_l}(f)_{|_E}+C_f(q_l)_{|_{E^\vee}}$. First, notice that for any $f=\sum_{r=1}^N \beta_r f_r \in E^\vee$ one gets 
	\[ C_{q_l}(f)  = \sum_{r=1}^N\sum_{i=1}^{l}\beta_r \big[f_r\cdot (e_{2i-1}\wedge e_{2i})\big] = \sum_{i=1}^{l}\big( \beta_{2i-1}e_{2i}-\beta_{2i}e_{2i-1}\big) \ , \]
	implying that $\Ker(\Phi_{q_l})\cap E^\vee=\{0\}$. Moreover, since $q_l\in \bigwedge^2E$, it is straightforward that
	\[ \bigwedge^{od \geq 5}E^\vee:= \bigoplus_{k\geq 2} \bigwedge^{2k+1}E^\vee \subset \Ker(\Phi_{q_l}) \ . \]
	As the images $\Phi_{q_l}(\bigwedge^{2k+1}E^\vee)$ are linearly independent one to each other, we get
	\[ \Ker(\Phi_{q_l})= \bigwedge^{od\geq 5}E^\vee \oplus \left( \Ker(\Phi_{q_l})\cap \bigwedge^3 E^\vee \right) \ , \]
	\noindent thus $\cal Z\left(\Ker(\Phi_{q_l})\right)$ is the intersection of the common zero loci of the two summands above.\\
	\indent Consider $[a] \in \cal Z\big(\Ker(\Phi_{q_l})\big)$ such that $a =\sum_{s=0}^{\frac{N}{2}}\sum_{J \in \binom{[N]}{2s}}\alpha_J\bold{e}_J$. From Lemma \ref{lemma:vanishing of global sections} it holds 
	\[ \langle V \cdot f,a\rangle =0 \ \ \ \ \ , \forall f \in \bigwedge^{od\geq 5}E^\vee\oplus \left( \Ker(\Phi_{q_l})\cap \bigwedge^3E^\vee \right) \ . \]
	The same computation in \eqref{vanishing with e's} shows that for any index subset $I\subset [N]$ such that $|I|=2k+1\geq 5$ and for any basis vector $e_\lambda \in E$ it holds $0=\langle e_\lambda \cdot \bold{f}_I, a\rangle = (-1)^{pos(\lambda, I)+1} \alpha_{I\setminus \{\lambda\}}$ where again $pos(\lambda,I)$ denotes the position of $\lambda$ in the ordered index subset $I$. We deduce that in $a$ there is no summand indexed by a subset $J$ such that $J\subset I$ and $|J|=|I|-1$ for any $|I|\geq 5$, hence $a \in \mathbb C \oplus \bigwedge^2E \oplus \bigwedge^{n+1}E$. On the other hand, the same computation as in \eqref{vanishing with f's} shows that for any $I \subset [N]$ such that $|I|=3$ and any basis vector $f_\nu \in E^\vee$ such that $\nu \notin I$, it holds $0=\langle f_\nu \cdot \bold{f}_I,a\rangle = (-1)^{pos(\nu,I\cup \{\nu\})+1}\alpha_{I\cup \{\nu\}}$, hence $a$ has no summand in $\bigwedge^{N}E$ either. It follows
	\[\cal Z\left(\Ker(\Phi_{q_l})\right)=\mathbb P\left(\mathbb C\oplus \bigwedge^2E\right) \cap \cal Z\left( \bigwedge^3E^\vee\cap \Ker\big(\Phi_{q_l}\big)\right) \ . \] 
	\indent Clearly, $[\mathbbm{1}]\in \cal Z(\Ker(\Phi_{q_l}))$ since for any $\lambda \in \mathbb C$ and any $f \in \bigwedge^3E^\vee$ one has $\langle f, v \cdot \lambda \rangle = \lambda \langle f, v \rangle =0$. Moreover, for any $b\in \bigwedge^2E$ it holds that $b \in \cal Z\left(\Ker(\Phi_{q_l})\right)$ if and only if $\lambda + b \in \cal Z\left(\Ker(\Phi_{q_l})\right)$ for any $\lambda \in \mathbb C$: in particular, it is enough to prove that $\cal Z\left(\Ker(\Phi_{q_l})\right)\cap \mathbb P(\bigwedge^2E)=\emptyset$ in order to conclude. \\
	\indent Let $[a]\in \cal Z\left(\Ker(\Phi_{q_l})\right)$ be such that $a=\sum_{\{s\lneq t\}\subset [N]}\alpha_{st}e_s\wedge e_t$. First, consider an index subset $I\subset [N]$ such that $|I|=3$ and $\{2k-1,2k\}\nsubseteq I$ for any $k\in [l]$: the condition $C_{\bold{f}_I}(q_l)_{|_{E^\vee}}=0$ implies $\bold{f}_I \in \Ker(\Phi_{q_l})$. Then, for any $e_\lambda \in E$ one gets 
	\[ 0= \langle e_\lambda \cdot \bold{f}_I,a\rangle = \delta_{\lambda, i_1}\alpha_{i_2,i_3} - \delta_{\lambda,i_2}\alpha_{i_1,i_3}+\delta_{\lambda, i_3}\alpha_{i_1,i_2} \ , \]
	\noindent for $\delta_{xy}$ being the Kronecker symbol. Since $N\geq 6$, given any two distinct indices $\{i\lneq j\}\subset [N]$ such that $\{i,j\}\neq \{2k-1,2k\}$ for any $k\in [l]$, one can always find a third index $r\in [N]\setminus \{i,j\}$ such that $\{2k-1,2k\}\nsubseteq I=\{i,j,r\}$ for any $k\in [l]$. Thus for any $\{i,j\}\neq \{2k-1,2k\}$ it holds $\alpha_{ij}=0$ and 
	\[\cal Z\left(\Ker(\Phi_{q_l})\right)\cap \mathbb P\left(\bigwedge^2E\right)\subset \left\{[a]=\left[\sum_{k=1}^{l}\alpha_{2k-1,2k}e_{2k-1}\wedge e_{2k}\right]\right\} \ . \]
	\noindent Now, for any $\{k\lneq h\}\subset [l]$ and any $r \in [N]\setminus\{2k-1,2k,2h-1,2h\}$ consider $\bold{f}_{2k-1,2k,r}-\bold{f}_{2h-1,2h,r} \in \bigwedge^3E^\vee$: it is a straightforward count that $\bold{f}_{2k-1,2k,r}-\bold{f}_{2h-1,2h,r} \in \Ker(\Phi_{q_l})$. In particular, $0=\langle e_r\cdot (\bold{f}_{2k-1,2k,r}-\bold{f}_{2h-1,2h,r}), e_r\cdot a\rangle= \alpha_{2k-1,2k}-\alpha_{2h-1,2h}$, hence it holds 
	\[ \alpha_{2k-1,2k}=\alpha_{2h-1,2h} \ \ \ \ \ , \forall \{h\lneq k\}\subset [l] \ . \]
	It follows that $a=\alpha\cdot q_l$ for some $\alpha \in \mathbb C$. But $[q_l] \notin \mathbb S_{N}^+$ (since $l\geq 3$), thus it has to be $\alpha=0$ and $\cal Z\left(\Ker(\Phi_{q_l})\right)\cap \mathbb P(\bigwedge^2E)=\emptyset$. We conclude that $\cal Z\left(\Ker(\Phi_{q_l})\right)=\{[\mathbbm{1}]\}$.
\end{proof}

\section{Dimensions of $\Spin_{2N}$--orbits in $\sigma_2(\mathbb S_N^+)$}\label{sec:orbit dimensions spinor}

\indent \indent In this section we compute the dimensions of each orbit in the secant variety of lines to a Spinor variety. We have postponed this computation as we apply tangential-identifiability for computing dimensions in the tangent branch. We assume the setting in Remark \ref{rmk:setting}. \\
\indent From \cite[Theorem 1.4]{zak1993tangents} the secant orbit $\Sigma_{\frac{N}{2},N}$ is dense of dimension $\dim \Sigma_{\frac{N}{2},N}=\dim \sigma_2(\mathbb S_{N}^+)$ as in \eqref{dim of sec}. Moreover, for $N\geq 6$ the tangential variety has codimension $1$, hence its dense orbit $\Theta_{\frac{N}{2},N}$ has dimension $\dim \Theta_{\frac{N}{2},N}=\dim \tau(\mathbb S_N^+)=N(N-1)$.\\
\indent Recall that for $N\leq 5$ the Spinor variety $\mathbb S_N^+$ has diameter $2$, and the distance--$2$ orbit $\Sigma_{2,N}$ either does not exist (for $N\leq 3$) or it is the dense one. In this respect, in the following we assume $N\geq 6$.

\begin{lemma}\label{lemma:fibers of theta as sec}
	For any $N\geq 6$ and any $l=2:\frac{N}{2}-1$, the fibration
	\[ \begin{matrix}
	\xi: & \Sigma_{l,N} & \longrightarrow & \OG(N-2l,V)\\
	& [a+b] & \mapsto & H_{a+b} & \stackrel{\text{Cor.} \ \ref{cor:fiber of Sigma_l}}{=}H_a\cap H_b
	\end{matrix} \]
	has fibers isomorphic to the dense orbit on the smaller Spinor variety $\mathbb S_{2l}^+$, namely $\xi^{-1}(H) \simeq \Sigma_{l,2l}\subset \sigma_2(\mathbb S_{2l}^+)$. In particular,
	\[\overline{\xi^{-1}(H)}\simeq \sigma_2(\mathbb S_{2l}^+) \ . \]
\end{lemma}
\begin{proof}
	Consider the distance-$l$ orbit $\Sigma_{l,N}=\Spin_{2N}\cdot [\bold{e}_{[N]} + \bold{e}_{[N-2l]}]$ and the above fibration $\xi$. By homogeneity, it is enough to determine the fiber at the subspace $\xi([\bold{e}_{[N]} + \bold{e}_{[N-2l]}])=E\cap E_{N-2l}=\langle e_1 , \ldots , e_{N-2l}\rangle_{\mathbb C}$. We set $V'= \langle e_{N-2l+1}, \ldots, e_N, f_{N-2l+1}, \ldots ,f_N \rangle_\mathbb C$ and $E'=V'\cap E$.\\
	\indent Let $[a+b]\in \xi^{-1}(E\cap E_{N-2l})$: then $E\cap E_l \subset H_a$ (resp. $E\cap E_l \subset H_b$) and, by definition of the maximal fully isotropic subspaces as kernels of \eqref{psi_a}, we can write $[a]=[\bold{e}_{[N-2l]} \wedge w_a ]$ (resp. $[b]=[\bold{e}_{[N-2l]}\wedge w_b]$) for some $w_a \in \bigwedge^{ev}E'$ (resp. $w_b \in \bigwedge^{ev}E'$). In particular, since $w_a$ and $w_b$ are defined by $2l$ linearly independent columns in the matrices describing $H_a$ and $H_b$, they corresponds to maximal isotropic subspaces $H_a', H_b' \subset \OG^+(2l,V')$, hence $[w_a],[w_b] \in \mathbb S_{2l}^+$. Finally, the condition $H_a\cap H_b=E\cap E_{N-2l}$ implies $H_a'\cap H_b'=(0)$, hence $[w_a]$ and $[w_b]$ have maximum Hamming distance $d([w_a],[w_b])=l$ in $\mathbb S_{2l}^+$. Therefore the injective wedge-multiplication map $(\bold{e}_{[N-2l]}\wedge \bullet) \ : \ \Sigma_{l,2l}\rightarrow \Sigma_{l,N}$ in \eqref{wedge-multiplication} gives the {\em biregular} isomorphism
	\begin{align*}
	\xi^{-1}(E\cap E_{N-2l}) & = \left\{[a+b] \in \Sigma_{l,N} \ | \ H_a\cap H_b = E\cap E_{N-2l}\right\} \\
	& = \left\{ [\bold{e}_{[N-2l]}\wedge (w_a+w_b)] \in \Sigma_{l,N} \ | \ [w_a],[w_b]\in \mathbb S_{2l}^+, \  d([w_a],[w_b])=l \right\} \\
	& \simeq \left\{ [w_a+w_b] \in \sigma_2^\circ(\mathbb S_{2l}^+) \ | \ d([w_a],[w_b])=l \right\} \\
	& = \Sigma_{l,2l} \ .
	\end{align*} 
\end{proof}

\begin{prop}\label{prop:orbit dimension in sec2}
	For any $N\geq 6$ and $l=2:\frac{N}{2}-1$, the secant orbit $\Sigma_{l,N} \subset \sigma_2^\circ(\mathbb S_{N}^+)$ has dimension
	\[
	\dim \Sigma_{l,N} = \begin{cases}
	\frac{N(N-1)}{2}+4N-15 & \text{if } l=2 \\
	\frac{N(N-1)}{2}+l(2N-1)-2l^2+1 & \text{if } l\geq3   \ .
	\end{cases} \]
\end{prop}
\begin{proof}
	Consider the fibration $\xi$ in Lemma \ref{lemma:fibers of theta as sec}. From the fiber dimension theorem we get
	\[ \dim \Sigma_{l,N}=\dim \OG(N-2l,V) + \dim \sigma_2(\mathbb S_{2l}^+) \ .\]
	In general, the orthogonal Grassmannian $\OG(r,M)$ coincides with the kernel of the global section $s_q\in H^0(\Gr(r,M),\Sym^2\cal U)$ induced by the quadratic form $q$ on $\mathbb C^M$, where $\cal U$ is the rank-$r$ universal bundle on $\Gr(r,M)$. Thus
	\[ \dim \OG(r,M) = \dim \Gr(r,M) - \rk \left(\Sym^2(\cal U)\right) = r(M-r)- \binom{r+1}{2} \ . \]
	The thesis follows by substituting $r=N-2l$ and $M=2N$, and recovering $\dim\sigma_2(\mathbb S_{2l}^+)$ from \eqref{dim of sec}.
\end{proof}

The above computation puts on light a particular feature of the second-to-last secant orbit. In the following result we need $N\geq 8$ in order to get an intermediate proper secant orbit between $\Sigma_{2,N}$ and $\Sigma_{\frac{N}{2},N}$.

\begin{cor}\label{prop:hypersurface in secant orbits}
	For any $N\geq 8$, the closure of the second-to-last orbit $\Sigma_{\frac{N-2}{2},N}$ is a divisor in $\sigma_2(\mathbb S_{N}^+)$ parametrized by the vanishing of a pfaffian. Indeed, up to chart-changing, all pure spinors $[a]\in \mathbb S_{N}^+$ such that $d([\bold{e}_{[N]}],[a])=\frac{N-2}{2}$ correspond to maximal fully isotropic subspaces $H_a$ described by matrices ${\tiny \begin{bmatrix}I_{N}\\A \end{bmatrix}}$ where $\rk(A)=N-\dim(E\cap H_a)=N-2$.
\end{cor} 
\begin{proof}
	The thesis \textquotedblleft being a divisor\textquotedblright \space is a straightforward count from Proposition \ref{prop:orbit dimension in sec2}. We show that the closure of $\Sigma_{\frac{N-2}{2},N}$ is parametrized by the vanishing of a pfaffian.\\
	\indent The pure spinors $[a]\in \mathbb S_{N}^+$ having Hamming distance $l$ from $[\bold{e}_{[N]}]$ correspond to subspaces $H_a\in \OG^+(N,V)$ such that $\dim (E\cap H_a)=N-2l$. Up to chart-changing, we may assume that $H_a$ is described by the matrix $\scriptsize{\begin{bmatrix} I_{N}\\ A \end{bmatrix}}$ for a certain $A \in \bigwedge^2\mathbb C^{N}$: in particular, $\rk(A)= N-\dim(E\cap H_a)= 2l$. It follows that the pure spinors having Hamming distance $l=\frac{N-2}{2}$ from $[\bold{e}_{[N]}]$ are described (up to chart-changing) by skew--symmetric matrices of rank $N-2$, that is they are parametrized by the vanishing of the pfaffian of such matrices. Now, given the pure spinor $[\bold{e}_{[N]}]$, the subvarieties
	\[\cal F_{\bold{e}_{[N]}, \frac{N-2}{2}}=\left\{[\bold{e}_{[N]}+ a] \ \big| \ [a]\in \mathbb S_{N}^+, \ d([\bold{e}_{[N]}],[a])\leq\frac{N-2}{2}\right\}\subset \overline{\Sigma_{\frac{N-2}{2},N}} \ , \]
	\[\cal F_{\bold{e}_{[N]}}=\left\{[\bold{e}_{[N]} + a] \ | \ [a]\in \mathbb S_{N}^+\right\}\subset \sigma_2^\circ(\mathbb S_{N}^+) \]
	are such that $\cal F_{\bold{e}_{[N]},\frac{N-2}{2}}= \cal F_{\bold{e}_{[N]}} \cap V(\Pf(A))$: more in general, for any pure spinor $[b]\in \mathbb S_{N}^+$ it holds $\cal F_{b,\frac{N-2}{2}}=\cal F_b \cap V(\Pf(A))$, giving the {\em rational} isomorphism in the thesis.
\end{proof}

\indent  We are left with computing dimensions in the tangent branch for $N\geq 6$. From Proposition \ref{prop:orbit dimension in sec2} we know that $\Theta_{2,N}=\Sigma_{2,N}$ has dimension $\frac{N(N-1)}{2}+4N-15$. Thus in the following we assume $N\geq 6$ and $l=3:\frac{N-1}{2}-1$.

\begin{obs*}
	In the following proof we use that the tangent orbits $\Theta_{l,N}$ for $l\geq 3$ are tangentially-identifiable (cf. Theorem \ref{thm:tangential-identifiable tangent orbit}): each point $[q]\in\Theta_{l,N}$ lies on a unique tangent space $T_{[a]}\mathbb S_{N}^+$. 
\end{obs*}

\begin{prop}\label{prop:orbit dimension in tau}
	For any $N\geq 6$ and any $l=3:\frac{N}{2}-1$, the tangent orbit $\Theta_{l,N} \subset \tau(\mathbb S_{N}^+)$ has dimension
	\begin{align*}
	\dim \Theta_{l,N} =\frac{N(N-1)}{2}+l(2N -1) -2l^2 \ . \end{align*}
	In particular, for such $N,l$, the closure $\overline{\Theta_{l,N}}$ is a divisor in $\overline{\Sigma_{l,N}}$.
\end{prop}
\begin{proof}
	Given $[\bigwedge^{2}\mathbb C^{N}]_{2l}$ the space of rank--$2l$ skew--symmetric matrices of size $N$, from Sec.\ \ref{subsec:tangent orbits} we know that for any $l=2:\frac{N}{2}$ it holds 
	\[\Theta_{l,N}=\bigcup_{[a]\in \mathbb S_{N}^+}\left[T_{[a]}\mathbb S_{N}^+\right]_{2l}\simeq \bigcup_{[a]\in \mathbb S_{N}^+}\left[\bigwedge^{2}\mathbb C^{N}\right]_{2l} \ . \]
	For $l\geq 3$, any tangent $[q]\in \Theta_{l,N}$ belongs to a unique tangent space $T_{[a]}\mathbb S_{N}^+$, hence
	\[ \dim \Theta_{l,N}=\dim \mathbb S_{N}^+ + \dim \left[\bigwedge^{2}\mathbb C^{N}\right]_{2l} \ \ \ \ \ , \ \forall l \geq 3 \ .\]
	The space $[\bigwedge^{2}\mathbb C^{N}]_{2l}$ is the $\GL(N)$-orbit in $\mathbb C^{N}\otimes \mathbb C^{N}$ of the skew--symmetric matrix $C_l$ in \eqref{standard skewsymm matrix} having stabilizer isomorphic to $\Sp(2l)\times \GL(N-2l) \times (\mathbb C^{2l}\otimes \mathbb C^{N-2l})$. Therefore one gets $\dim \left[\bigwedge^2\mathbb C^{N}\right]_{2l} = l(2N-1)-2l^2$
	and the thesis straightforwardly follows.
\end{proof}

\begin{es}\label{S6 orbits} The Spinor variety $\mathbb S_{6}^+ \subset \mathbb P(\bigwedge^{ev}\mathbb C^6)\simeq \mathbb P^{31}$ has diameter $3$. Set $G=\Spin(12)$ and $v_{\omega_6}=e_1\wedge \ldots \wedge e_6$. The secant variety $\sigma_2(\mathbb S_6^+)$ stratifies in the $\Spin(12)$--orbit closures \[\mathbb S_6^+ \subset \overline{\Sigma_{2,6}}\subset \overline{\Theta_{3,6}}\subset \overline{\Sigma_{3,6}}=\sigma_2(\mathbb S_6^+)\] of dimensions respectively $15, 24, 30, 31$. This poset and the dimensions were already known to J.M. Landsberg and L. Manivel \cite{LM01} in the context of Legendrian varieties. Actually, as confirmed by the authors (which we thank for the confrontation), our arguments allow to recognize a misprint in \cite[Proposition 5.10]{LM01}, where the dimension of the orbit $\sigma_+$ (in the authors' notation, corresponding to our $\Sigma_{2,6}$ for $m=4$) is $5m+4$ instead of $5m+3$.
\end{es}

\begin{es}\label{S8 orbits} The Spinor variety $\mathbb S_8^+\subset \mathbb P(\bigwedge^{ev}\mathbb C^8)\simeq \mathbb P^{127}$ has diameter $4$: this is the only case in which the inclusion $\overline{\Theta_{3,8}}\subset \overline{\Theta_{4,8}} \cap \overline{\Sigma_{3,8}}$ actually is an equality.
	\begin{figure}[H]
		\begin{center}
			{\small \begin{tikzpicture}[scale=2.5]
				
				\node(S) at (0,0){{$\mathbb S_{8}^+$}};
				\node(dimS) at (-1.5,0){{$28$}};
				
				\node(t2) at (0,0.4){{$\Theta_{2,8}=\Sigma_{2,8}$}};
				\node(dimt2) at (-1.5,0.4){{$45$}};
				
				\node(t3) at (-0.6,0.8){{$\Theta_{3,8}$}};
				\node(dimt3) at (-1.5,0.8){{$55$}};
				
				\node(s3) at (0.6,1.2){{$\Sigma_{3,8}$}};
				\node(dims3) at (-1.5,1.2){{$56$}};
				
				\node(t4) at (-0.6,1.2){$\Theta_{4,8}$};
				
				\node(s4) at (0.6,1.6){{$\Sigma_{4,8}$}};
				\node(dimsd) at (-1.5,1.6){{$57$}};
				
				\path[font=\scriptsize,>= angle 90]
				(S) edge [->] node [left] {} (t2)
				(t2) edge [->] node [left] {} (t3)
				(t3) edge [->] node [left] {} (s3)
				(t3) edge [->] node [left] {} (t4)
				(t4) edge [->] node [left] {} (s4)
				(s3) edge [->] node [left] {} (s4);
				\end{tikzpicture}}
		\end{center}
		\caption{Poset graph of the $\Spin_{16}$--orbits in $\sigma_2(\mathbb S_{8}^+)$, and their dimensions.} \label{figure:graph for S8}
	\end{figure}
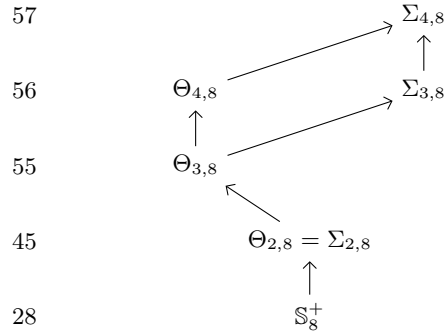 
\end{es}

\section{The $2$--nd Terracini locus of $\mathbb S_N^+$}\label{sec:terracini locus spinor}

\indent \indent In the following we determine the {\em Terracini locus}, introduced in \cite{ballico2021terracini,ballico2020terracini}. Its importance relies in the fact that it gives information on the singularities of points lying on bisecant lines.\\
\indent The {\em $2$--nd Terracini locus} of $\mathbb S_N^+$ is 
\[ \Terr_2(\mathbb S_N^+):=\overline{\left\{ \{p_1,p_2\}\in Hilb_2(\mathbb S_N^+) \ | \ \dim \langle T_{p_1}\mathbb S_N^+,T_{p_2}\mathbb S_N^+ \rangle \lneq \dim \sigma_2(\mathbb S_N^+)\right\}} \ . \]
We also recall the definition of the {\em $2$--nd abstract secant variety} of $\mathbb S_N^+$
\[ Ab\sigma_2(\mathbb S_N^+):=\left\{(\cal Z,p)\in Hilb_2(\mathbb S_N^+)\times \mathbb P\left(\bigwedge^{ed}E\right) \ | \ p \in \langle \cal Z \rangle\right\} \ , \]
together with the natural projections onto the two factors $\pi_1:Ab\sigma_2(\mathbb S_N^+)\rightarrow Hilb_2(\mathbb S_N^+)$ and $\pi_2:Ab\sigma_2(\mathbb S_N^+)\rightarrow \sigma_2(\mathbb S_N^+)\subset \mathbb P(\bigwedge^{ev}E)$. Notice that, with this definition, the $2$--nd abstract secant variety is smooth. \\
\indent The $2$--nd Terracini locus tells us where the differential of the projection from the abstract secant variety onto the secant variety drops rank: this information, combined with the identifiability, allows to deduce the smoothness on the locus of points which are both identifiable and outside the Terracini locus. In the following we consider $N\geq 7$ as for $N\leq 6$ the secant variety of lines coincides with the ambient space.

\begin{teo}\label{thm:terracini locus spinor}
	For any $N\geq 7$, the second Terracini locus $\Terr_2(\mathbb S_N^+)$ of the Spinor variety $\mathbb S_N^+$ corresponds to the distance--$2$ orbit closure $\overline{\Sigma_{2,N}}$. More precisely, in the above notation, it holds
	\[ \Terr_2(\mathbb S_N^+) = \left( \pi_1 \circ \pi_2^{-1}\right) \left( \overline{\Sigma_{2,N}} \right) \ . \]
\end{teo}
\begin{proof}
	Given a point $[a+b]\in \sigma_2(\mathbb S_N^+)$ for certain $[a],[b]\in \mathbb S_N^+$, we show that $\dim \langle T_{[a]}\mathbb S_N^+,T_{[b]}\mathbb S_N^+ \rangle$ drops if and only if $[a+b]\in \overline{\Sigma_{2,N}}=\mathbb S_N^+\sqcup \Sigma_{2,N}$, that is if and only if $d([a],[b])\leq 2$. \\
	\indent If $d([a],[b])=1$, then any point on the line $L([a],[b])\subset \mathbb S_N^+$ is direction of the curve defined by such line, that is $L([a],[b])\subset T_{[a]}\mathbb S_N^+\cap T_{[b]}\mathbb S_N^+$.\\
	\indent If $d([a],[b])=2$, then there exists $[c]\in \mathbb S_N^+$ such that $L([a],[c]), L([c],[b])\subset \mathbb S_N^+$. In particular, from the previous case we deduce $[c]\in T_{[a]}\mathbb S_N^+\cap T_{[b]}\mathbb S_N^+$.\\
	\indent On the other hand, if the dimension drops, then there exists a common non-zero tangent point $[x]\in T_{[a]}\mathbb S_N^+\cap T_{[b]}\mathbb S_N^+$. In particular, $[x]$ is not tangential-identifiable, hence $[x] \in \Theta_{2,N}$. From the definition \eqref{def:tangent orbits} of $\Theta_{2,N}\cap T_{[a]}\mathbb S_N^+$ as the set of rank--$4$ skew-symmetric matrices in $T_{[a]}\mathbb S_N^+$, we know that $[x]=[y+z]$ for $[y],[z]\in \Theta_{1,N}\cap T_{[a]}\mathbb S_N^+=\mathbb S_N^+ \cap T_{[a]}\mathbb S_N^+$ skew-symmetric matrices of rank $2$. In particular, $[a+y]$ and $[a+z]$ lie in $\mathbb S_N^+$, that is $d([a],[y])=d([a],[z])=1$. The same argument shows that also $d([b],[y])=d([b],[z])=1$. We conclude that $d([a],[b])\leq 2$, hence the thesis.

\end{proof}

\section{Results on the singular locus of $\sigma_2(\mathbb S_{N}^+)$}\label{sec:sing locus spinor}

\indent \indent This section is devoted to study the singular locus of the secant variety of lines $\sigma_2(\mathbb S_{N}^+)$ to a Spinor variety $\mathbb S_{N}^+$. We use results from previous sections.

\begin{obs*}
	As defective cases, the secant varieties $\sigma_2(\mathbb S_4^+)=\tau(\mathbb S_4^+)=\mathbb P^{7}$ and $\sigma_2(\mathbb S_{5}^+)=\tau(\mathbb S_5^+)=\mathbb P^{15}$ overfill the ambient space. On the other hand, the secant variety $\sigma_2(\mathbb S_6^+)=\mathbb P^{31}$ perfectly fills the ambient space (hence it is smooth), but the tangential variety $\tau(\mathbb S_6^+)$ is a quartic hypersurface in it. Accordingly to Remark \ref{rmk:S_6 tangential-identifiability}, $\mathbb S_6^+$ has been widely studied in \cite{LM01, LM07} and it has been proven that $\Sing(\tau(\mathbb S_6^+))=\overline{\Sigma_{2,6}}$ (corresponding to $\sigma_+$ in \cite{LM07}).
\end{obs*}

\indent According to the above remark, we assume $N\geq 7$, so that $\sigma_2(\mathbb S_{N}^+)\subsetneq\mathbb P^{2^{N-1}-1}$. In the following we consider the alternative definition of abstract secant variety
\[ Ab\sigma_2(\mathbb S_N^+):= \overline{\left\{ ([a],[b],[q])\in \big( \mathbb S_N^+ \big)^2_{/\mathfrak{S}_2} \times \mathbb P\left( \bigwedge^{ev}E\right) \ \bigg| \ [q]\in \langle [a],[b]\rangle \right\}} \ , \]
which is smooth only outside the preimage of the diagonal $\Delta_X\subset \big( \mathbb S_N^+ \big)^2_{/\mathfrak{S}_2}$ via the projection onto the first factor. In particular, given $\pi$ the projection onto the second factor, the preimage $\pi^{-1}(\sigma_2(\mathbb S_N^+)\setminus \overline{\Sigma_{2,N}})$ is smooth in $Ab\sigma_2(\mathbb S_N^+)$.

\begin{lemma}\label{lemma:singularity of Sigma_2}
	For any $N\geq 7$, the distance--$2$ orbit $\Sigma_{2,N}$ lies in the singular locus $\Sing(\sigma_2(\mathbb S_N^+))$ of the secant variety of lines to the Spinor variety $\mathbb S_{N}^+$. In particular, 
	\[ \overline{\Sigma_{2,N}}\subset \Sing\left(\sigma_2(\mathbb S_{N}^+)\right) \ .\]
\end{lemma}
\begin{proof}
	Clearly, the Spinor variety is singular in its secant variety of lines, hence we just have to prove the singularity of the distance--$2$ orbit $\Sigma_{2,N}$. We assume by contradiction that $\Sigma_{2,N}$ is smooth: then the poset in Theorem \ref{thm:orbit partition of sec} implies that all the open subset $\sigma_2(\mathbb S_N^+)\setminus \mathbb S_N^+$ is smooth. Consider the projection from the abstract secant variety 
	onto the second factor
	\[ \begin{matrix}
	\pi: & Ab\sigma_2(\mathbb S_{N}^+) & \longrightarrow & \sigma_2(\mathbb S_{N}^+) \\
	& ([a],[b],[q]) & \mapsto & [q]	
	\end{matrix} \ . \]
	\indent The (tangential-)identifiability of the orbits $\Sigma_{l,N}$ and $\Theta_{l,N}$ for $l\geq 3$ (cf. Theorem \ref{thm:identifiable secant orbits} and Theorem \ref{thm:tangential-identifiable tangent orbit}) implies that the restriction of $\pi$ to the open subset $\pi^{-1}(\sigma_2(\mathbb S_N^+)\setminus \overline{\Sigma_{2,N}})$ is a bijection of smooth open subsets, hence it is an isomorphism. 
	On the other hand, the orbit $\Sigma_{2,N}$ is unidentifiable and any of its points has $6$--dimensional decomposition locus (cf. Proposition \ref{prop:contact locus for Sigma2}), thus the differential $d(\pi)_{([a],[b],[q])}$ drops rank exactly at the points in the preimage $\pi^{-1}(\overline{\Sigma_{2,N}})$. \\
	\indent It follows that the restriction  
	\[{\pi}_{|}: Ab\sigma_2(\mathbb S_{N}^+)\setminus \pi^{-1}(\mathbb S_{N}^+)\longrightarrow \sigma_2(\mathbb S_{N}^+)\setminus \mathbb S_{N}^+ \]
	is a morphism of smooth varieties of the same dimension: in particular, the locus of points where the rank of the differential drops is a (determinantal) divisor. We show that the preimage
	\[\left\{([a],[b],[q])\in Ab\sigma_2(\mathbb S_{N}^+)\setminus \pi^{-1}(\mathbb S_{N}^+) \ | \ \rk d(\pi)_{([a],[b],[q])}< N(N-1)\right\}={\pi_{|}}^{-1}\left(\overline{\Sigma_{2,N}}\right) \]
	cannot be a divisor (leading to a contradiction). Indeed, from Proposition \ref{prop:orbit dimension in sec2} we know that $\dim \Sigma_{2,N}=\frac{N(N-1)}{2}+4N-15$, and from Proposition \ref{prop:contact locus for Sigma2} that the decomposition locus of any point in $\Sigma_{2,N}$ is $6$--dimensional, thus from the fiber dimension theorem we get $\dim\pi_{|}^{-1}(\overline{\Sigma_{2,N}})=\dim \Sigma_{2,N}+6=\frac{N(N-1)}{2}+4N-9$ and 
	\[\dim Ab\sigma_2(\mathbb S_{N}^+) - \dim\pi_{|}^{-1}(\overline{\Sigma_{2,N}}) = \frac{N(N-1)}{2}-4N +10 \stackrel{N\geq 7}{\gneq} 1 \ . \]
	Thus ${\pi}_{|}^{-1}(\overline{\Sigma_{2,N}})$ cannot be a divisor for $N\geq7$, giving the contradiction. We deduce the inclusion $\Sigma_{2,N}\subset \Sing(\sigma_2(\mathbb S_{N}^+))$.
\end{proof}

\begin{lemma}\label{lemma:smoothness secant orbit spinor}
	For any $N\geq 7$ and any $l=3:\frac{N}{2}$, the secant orbit $\Sigma_{l,N}$ is smooth in the secant variety of lines $\sigma_2(\mathbb S_N^+)$. In particular, the singular locus of the secant variety of lines to the Spinor variety $\mathbb S_N^+$ lies in the tangential variety:
	\[ \Sing\left(\sigma_2(\mathbb S_N^+)\right) \subset \tau(\mathbb S_N^+) \ . \]
\end{lemma}
\begin{proof}
	From Proposition \ref{thm:identifiable secant orbits} we know that the secant orbits $\Sigma_{l,N}$ for $l=3:N$ are identifiable. Moreover, from the poset in Figure \ref{figure:graph} it holds
	\[\Sigma_{3,N}\sqcup \ldots \sqcup \Sigma_{\frac{N}{2},N}=\sigma_2(\mathbb S_N^+)\setminus \tau(\mathbb S_N^+) \ . \] 
	Then the projection $\pi:Ab\sigma_2(\mathbb S_N^+)\rightarrow \sigma_2(\mathbb S_N^+)$ from the abstract secant variety onto the second factor restricts to a bijection
	\[ \pi': \pi^{-1}\bigg(\sigma_2(\mathbb S_N^+)\setminus \tau(\mathbb S_N^+)\bigg)\rightarrow \sigma_2(\mathbb S_N^+)\setminus \tau(\mathbb S_N^+) \ . \]
	Finally, the points in $\sigma_2(\mathbb S_N^+)\setminus \tau(\mathbb S_N^+)$ are outside the Terracini locus $\Terr_2(\mathbb S_N^+)$ (cf. Theorem \ref{thm:terracini locus spinor}), thus for any point $[p+q]\in \sigma_2(\mathbb S_N^+)\setminus \tau(\mathbb S_N^+)$ it holds $T_{p}\mathbb S_N^+ \cap T_{q}\mathbb S_N^+=\{0\}$. In particular, for such points the differential $d\pi'_{([p],[q],[p+q])}$ maps $T_p\mathbb S_N^+\times T_q\mathbb S_N^+$ to $T_p\mathbb S_N^+\oplus T_q\mathbb S_N^+$, hence it is injective. It follows that $\pi'$ is an isomorphism and the open subset $\sigma_2(\mathbb S_N^+)\setminus \tau(\mathbb S_N^+)$ is smooth in the secant variety of lines.
\end{proof}

\indent Collecting Lemma \ref{lemma:singularity of Sigma_2} and Lemma \ref{lemma:smoothness secant orbit spinor} we get the following bound on the singular locus.
\begin{cor}\label{cor:sing locus spinor}
	For any $N\geq 7$, the singular locus of the secant variety of lines $\sigma_2(\mathbb S_N^+)$ to the Spinor variety $\mathbb S_N^+$ is such that
	\[ \overline{\Sigma_{2,N}} \subset \Sing\left( \sigma_2(\mathbb S_N^+)\right) \subset \tau(\mathbb S_N^+) \ . \]
	\end{cor}
We conclude this work with a conjecture on what the singular locus of $\sigma_2(\mathbb S_N^+)$ actually is, motivated by the arguments apprearing in the appendix. 

\begin{conjecture}\label{conj:sing locus spinor}
	For any $N\geq 7$, the singular locus of the secant variety of lines to the Spinor variety $\mathbb S_N^+$ is the closure of the distance--$2$ orbit, i.e.
	\[ \Sing\left(\sigma_2(\mathbb S_N^+)\right)=\overline{\Sigma_{2,N}}=\mathbb S_N^+ \sqcup \Sigma_{2,N} \ . \]
\end{conjecture}

\section*{Appendix}

\indent \indent The arguments in this section are inspired from \cite{vermeire2001some, vermeire2009singularities, ullery}: the author thanks L. Manivel for suggesting these references. They have been investigated in the attempt of proving the smoothness of both secant and tangent orbits of distance greater or equal than $3$ in the secant variety of lines to a Spinor variety. The idea was to prove that the secant bundle gives a desingularization of the secant variety, but we have just been able to exhibit a bijection (as sets) between dense subsets.\\
\hfill\break
\indent We recall that a line bundle $\cal L$ on a smooth projective variety $X$ is {\em $k$--very ample} if every $0$--dimensional subscheme $\cal Z\subset X$ of length $k+1$ imposes independent conditions on $\cal L$, or equivalently if the restriction map $H^0(X,\cal L)\rightarrow H^0(X,\cal L \otimes \cal O_{\cal Z})$ is surjective: in particular, $\cal L$ is $1$--very ample if and only if it is very ample, and $\cal L$ is $k$--very ample if and only if there does not exist a $0$--dimensional subscheme $\cal Z\subset X$ of length $k+1$ lying on a linear subspace $\mathbb P^{k-1}\subset \mathbb P(H^0(X,\cal L)^\vee)$.

\begin{obs}\label{rmk:O(1) is not 2-very ample}
	The line bundle $\cal O_{\mathbb S_N^+}(1)$ on the Spinor variety $\mathbb S_N^+$ is very ample, giving the embedding in $\mathbb P\left( H^0(\mathbb S_{N}^+,\cal O_{\mathbb S_N^+}(1))^\vee \right)\simeq \mathbb P\left(\bigwedge^{ev}E \right)$, but not $2$--very ample. Indeed, it would be $2$--very ample if and only if there would not exist a $0$--dimensional subscheme of $\mathbb S_N^+$ of length $3$ lying on a line $\mathbb P^1\subset \mathbb P(\bigwedge^{ev}E)$. But we know that the Spinor variety contains lines, defined by pairs of pure spinors having Hamming distance $1$.
\end{obs}

\paragraph*{Secant bundle.} Let $\cal L=\cal O(1)$ be the very ample line bundle on $X:=\mathbb S_{N}^+$, and let $Hilb_2(X)$ be the Hilbert scheme of $0$-dimensional subschemes of $\mathbb S_{N}^+$ of length $2$. We denote by $\cal E_\cal L$ the locally free sheaf of rank $2$ on $Hilb_2(X)$ having fibers $\left( \cal E_\cal L\right)_\cal Z = H^0(X,\cal L \otimes \cal O_\cal Z)$: formally, $\cal E_\cal L:=(\pi_2)_\ast(\pi_1^\ast)(\cal L)$ where $\pi_1$ and $\pi_2$ are the natural projections from the universal family $\Phi:=\{(x,\cal Z) \in X\times Hilb_2(X) \ | \ x \in \cal Z\}$ onto $X$ and $Hilb_2(X)$ respectively.
From the projection formula, one gets the following global sections of sheaves:
\begin{align*}
H^0(Hilb_2(X),\cal E_\cal L) & = H^0(\Phi, (\pi_X)^\ast \cal L) = H^0( X, \cal L \otimes (\pi_X)_\ast\cal O_\Phi ) = H^0( X,\cal L) \ ,
\end{align*}
where the last equality follows since $\pi_X$ is a proper projection and $(\pi_X)_\ast\cal O_\Phi=\cal O_X$.\\
By pushing forward via $\pi_H$ the morphism $H^0(X,\cal L) \otimes \cal O_\Phi \rightarrow (\pi_X)^\ast \cal L$ (of sheaves on $\Phi$) one gets the {\em evaluation} morphism of sheaves on $Hilb_2(X)$
\[ ev: H^0(X,\cal L)\otimes \cal O_{Hilb_2(X)} \longrightarrow \cal E_\cal L \ , \]
defined on the fibre as the restriction: for any $\cal Z\in Hilb_2(X)$ it holds
\[\begin{matrix}
ev_\cal Z: & H^0(X,\cal L) & \longrightarrow & \left(\cal E_\cal L\right)_\cal Z & = H^0(X,\cal L \otimes \cal O_\cal Z) \\
& s & \mapsto & s_{|_\cal Z}
\end{matrix} \ . \]
\indent As $\cal L$ is very ample, the evaluation 
\[
ev_\cal Z: H^0(X,\cal L)\twoheadrightarrow H^0(X,\cal \otimes \cal O_\cal Z)
\]
is a surjection for any $\cal Z \in Hilb_2(X)$, hence the morphism of sheaves $H^0(X,\cal L)\otimes \cal O_{Hilb_2(X)}\rightarrow \cal E_\cal L$ is surjective. Moreover, again by very ampleness of $\cal L$, one can get the injective morphism
\[ \begin{matrix}
\psi: & Hilb_2(X) & \hookrightarrow & \Gr\big(2,H^0(X,\cal L)\big)\\
& \cal Z & \mapsto & \left[ H^0(X,\cal L)\twoheadrightarrow H^0(X,\cal L \otimes \cal O_\cal Z) \right]
\end{matrix} \ . \]
Now, we may think at $\mathbb P^M=\mathbb P(H^0(X,\cal L)^\vee)$ as to the $1$-dimensional quotients of $H^0(X,\cal L)$, i.e. 
\[ \mathbb P^M=\left\{H^0(X,\cal L)\twoheadrightarrow Q \ | \ \dim Q=1\right\} \ .\]
We denote by $Q^1$ a $1$-dimensional vector space. Then one defines the (first) {\em secant bundle} $\mathbb P\cal E_\cal L$ as the $\mathbb P^1$-bundle on $Hilb_2(X)$ with fibers $(
\mathbb P\cal E_\cal L)_{\cal Z} = \left\{\left(\cal Z, \left[H^0(X,\cal L \otimes \cal O_{\cal Z})\twoheadrightarrow \cal Q^1\right]\right)\right\}$. The following morphism is well-defined:
\[ \begin{matrix}
f: & \mathbb P\cal E_\cal L & \longrightarrow & \mathbb P\left(H^0(X,\cal L)^\vee\right) \\
& \left(\cal Z, \left[H^0(X,\cal L \otimes \cal O_\cal Z)\twoheadrightarrow Q^1\right]\right) & \mapsto & \left[ H^0(X,\cal L)\twoheadrightarrow Q^1 \right]
\end{matrix} \ .\]
Notice that $\left[H^0(X,\cal L)\twoheadrightarrow Q^1\right]\subset \im(f)$ if and only if it factors through a certain $H^0(X,\cal L\otimes \cal O_\cal Z)$ for some $\cal Z \in Hilb_2(X)$: in particular, 
\[ f\left( \mathbb P\cal E_\cal L \right)=\sigma_2(X)\subset \mathbb P(H^0(X,\cal L)^\vee) \ . \]
\indent For any $\cal Z\in Hilb_2(X)$, the fibre $\left( \mathbb P\cal E_\cal L\right)_\cal Z=\left\{(\cal Z, \left[H^0(X,\cal L\otimes \cal O_\cal Z)\twoheadrightarrow Q^1\right])\right\}$ is bijectively mapped via $f$ to the secant line $L(\cal Z)\subset \sigma_2(X)$ intersecting $X$ in the points of $\cal Z$: if $\cal Z=\{p,q\}$ is a reduced subscheme, then the fibre at $\cal Z$ gives the bisecant line $\left(\mathbb P \cal E_\cal L \right)_\cal Z \stackrel{1:1}{\mapsto} L(p,q)$, while if $\cal Z$ is a non-reduced subscheme corresponding to $\{p,v\}$, where $p \in X$ and $v \in T_pX$, then the fibre at $\cal Z$ gives the tangent line $\left(\mathbb P \cal E_\cal L \right)_\cal Z \stackrel{1:1}{\mapsto} L_v(p)=p+tv$. Finally, in our setting, for any two distinct subschemes $\cal Z, \cal Y \in Hilb_2(X)$ the lines $L(\cal Z)$ and $L(\cal Y)$ coincide if and only if they lie in $X$, as the Spinor variety is intersection of quadrics.

\begin{lemma}\label{lemma:singular locus lies in distance-2}
	The restriction of $f$ to $f^{-1}(\sigma_2(\mathbb S_N^+)\setminus \overline{\Sigma_{2,N}})$ is a bijection.
\end{lemma}
\begin{proof}
	By Theorem \ref{thm:identifiable secant orbits}, the points in the secant orbits $\Sigma_{l,N}$ for $l\geq 3$ are identifiable, thus any two bisecant lines (given by points in $\Sigma_{l,N}$ for $l\geq 3$, except for the two points in which they intersect $\mathbb S_{N}^+$) do not intersect away from $\mathbb S_{N}^+$. By Theorem \ref{thm:tangential-identifiable tangent orbit}, the points in the tangent orbits $\Theta_{l,N}$ for $l\geq 3$ are tangential-identifiable, liying on a unique tangent line, hence any two tangent lines do not intersect away from $\overline{\Sigma_{2,N}}$. Finally, by the orbit partition of $\sigma_2(\mathbb S_{N}^+)$ in Theorem \ref{thm:orbit partition of sec}, since $\Theta_{l,N}\neq \Sigma_{l,N}$ for any $l\geq 3$, a tangent line and a bisecant line never intersect away from $\overline{\Sigma_{2,N}}$. It follows that $f: \mathbb P\cal E_\cal L\rightarrow \sigma(\mathbb S_{N}^+)$ is a bijection away from $f^{-1}\left(\overline{\Sigma_{2,N}} \right)$.
\end{proof}

\begin{obs*}
	We point out that Vermeire in \cite[Proposition 1.2, Theorem 2.2]{vermeire2009singularities} and Ullery in \cite[Lemma 1.1]{ullery} prove the smoothness of $\sigma_2(X)\setminus X$ under the assumption of the line bundle giving the embedding to be $3$-very ample: this assumption is the one {\em separating secant and tangent lines} away from $X$. In our case, the previous hypothesis is not satisfied (see Remark \ref{rmk:O(1) is not 2-very ample}). Neverthless, at a set-level, the identifibiability and the tangential-identifiability allow to separate secant and tangent lines (away from $\overline{\Sigma_{2,N}}$) as well. What is missing is that the above restriction has injective differential, so that it would be an isomorphism.
\end{obs*}

\printbibliography

\end{document}